\documentclass[11pt,english]{amsart}
\usepackage[T1]{fontenc}
\usepackage[latin9]{inputenc}
\usepackage{amsthm}
\usepackage{amstext}
\usepackage{stmaryrd}
\usepackage{cite}
\usepackage{esint}
\usepackage{mathrsfs}
\makeatletter
\numberwithin{equation}{section}
\numberwithin{figure}{section}
\theoremstyle{plain}
\newtheorem{thm}{\protect\theoremname}[section]
  \theoremstyle{plain}
  \newtheorem{lem}[thm]{\protect\lemmaname}

\usepackage{amsmath,amsfonts,amssymb,amsthm,epsfig}

\usepackage{color}

\usepackage[leqno]{amsmath}

\newcommand{\ds}{\displaystyle}

\def\R{\mathbb R}

\def\ga{\gamma}

\def\la{\lambda}

\def\var{\varphi}

\def\Om{\Omega}

\def\pa{\partial}

\numberwithin{equation}{section}
\newtheorem{theorem}{Theorem}[section]

\theoremstyle{definition}
\begingroup
\newtheorem{rem}[thm]{Remark}
\newtheorem{prop}[thm]{Proposition}

\newtheorem{remark}[theorem]{Remark}

\endgroup

\makeatother

\usepackage{babel}
  \providecommand{\lemmaname}{Lemma}
\providecommand{\theoremname}{Theorem}

\begin{document}
\title[Solutions for a nonlocal elliptic equation]
{Solutions for a nonlocal elliptic equation involving critical growth and Hardy potential}

\author{ Chunhua Wang, Jing Yang, Jing Zhou}

\date{}

\address{[Chunhua Wang] School of Mathematics and Statistics, Central China
Normal University, Wuhan 430079, P. R. China. }

\email{[Chunhua Wang] chunhuawang@mail.ccnu.edu.cn}

\address{[Jing Yang] College of mathematics and physics, Jiangsu University of Science and
Technology,Zhenjiang 212003, P. R. China.}

\email{[Jing Yang] yyangecho@163.com}

\address{[Jing Zhou] School of Mathematics and Statistics, South-Central University for Nationalities, Wuhan 430074, P. R. China }

\email{[Jing Zhou] zhouj@mail.scuec.edu.cn}

\begin{abstract}
In this paper, by an approximating argument,  we obtain infinitely
many solutions for the following Hardy-Sobolev fractional equation with
critical growth
\begin{equation*}\label{0.1}
\left\{%
\begin{array}{ll}
    (-\Delta)^{s} u-\ds\frac{\mu u}{|x|^{2s}}=|u|^{2^*_s-2}u+au, & \hbox{$\text{in}~ \Omega$},\vspace{0.1cm} \\
   u=0,\,\, &\hbox{$\text{on}~\partial \Omega$}, \\
\end{array}%
\right.
\end{equation*}
provided $N>6s$, $\mu\geq0$, $0< s<1$, $2^*_s=\frac{2N}{N-2s}$, $a>0$
is a constant and $\Omega$ is an open bounded domain in $\R^N$
which contains the origin.

\end{abstract}

\maketitle
{\small
\smallskip

%\tableofcontents{}

\section{Introduction}
Let $0< s<1$, $N> 6s, 2^*_s=\frac{2N}{N-2s},$ $\bar{\mu}$ be defined later, and $\Omega$ be
an open bounded domain in $\R^{N}$ which contains the origin.
  We study the following nonlinear fractional problem
\begin{equation}\label{1.1}
\left\{%
\begin{array}{ll}
    (-\Delta)^{s} u-\ds\frac{\mu u}{|x|^{2s}}=|u|^{2^*_s-2}u+au, & \hbox{$\text{in}~ \Omega$},\vspace{0.1cm} \\
   u=0,\,\, &\hbox{$\text{on}~\partial \Omega$}, \\
\end{array}%
\right.
\end{equation}
where $\mu\geq0$ satisfies $\frac{2^*_s\sqrt{\bar{\mu}}}{\sqrt{\bar{\mu}}-\sqrt{\bar{\mu}-\mu}}>\frac{2N}{N-6s}$,
$a>0$ is a constant, and $(-\Delta)^{s}$ stands for the fractional  Laplacian
operator in $\Om$ with zero Dirichlet boundary values on $\partial \Om$.

To define the fractional  Laplacian operator $(-\Delta)^{s}$ in $\Om$, let $\{\la_k,\var_k\}$ be the eigenvalues and corresponding eigenfunctions of the
Laplacian operator $-\Delta$ in $\Om$ with zero Dirichlet boundary values on $\partial \Om$,
\begin{equation*}\label{1.1.1}
\left\{%
\begin{array}{ll}
    -\Delta \var_k=\la_k\var_k, & \hbox{$\text{in}~ \Omega$},\vspace{0.1cm} \\
   \var_k=0,\,\, &\hbox{$\text{on}~\partial \Omega$}, \\
\end{array}%
\right.
\end{equation*}
normalized  by $\|\var_k\|_{L^2(\Om)}=1$. Then one can define $(-\Delta)^{s}$ for $s\in(0,1)$ by
$$
(-\Delta)^{s}u=\sum_{k=1}^\infty \la^s_kc_k\var_k,
$$
which clearly maps
$$
H^s_0(\Om):=\Big\{u=\sum_{k=1}^\infty c_k\var_k\in L^2(\Om): \sum_{k=1}^\infty \la^s_kc^2_k<\infty\Big\}
$$
into $L^2(\Om)$. Moreover, $H_0^s(\Om)$ is a Hilbert space equipped with an inner product
$$
\Big \langle \sum_{k=1}^\infty c_k\var_k,\sum_{k=1}^\infty d_k\var_k\Big\rangle_{H_0^s(\Om)}
=\sum_{k=1}^\infty \la^s_kc_kd_k,\,\,\,\text{if}\,\,\sum_{k=1}^\infty c_k\var_k,\sum_{k=1}^\infty d_k\var_k\in H^s_0(\Om).
$$

Now we can write the functional corresponding to \eqref{1.1} as
\begin{equation}\label{1.1.2}
I(u)=\frac{1}{2}\ds\int_{\Om}\Big(\big|(-\Delta)^{\frac{s}{2}} u\big|^2-\mu\frac{u^{2}}{|x|^{2s}}-au^2\Big)dx
-\frac{1}{2^*_s}\ds\int_{\Om}|u|^{2^*_s}dx,\,\,\,\,u\in H^s_0(\Om).
\end{equation}

A great deal of work has currently been devoted to the study of the fractional Laplacian operator
 as it appears in several applications to some models related to anomalous diffusions in plasmas,
flames propagation and chemical reactions in liquids, population dynamics, geophysical fluid
dynamics, and American options in finance, see, e.g., \cite{abfs,es,fl,GM,kmr,A} and the references therein.
 We refer the reader to
\cite{bbc,cx,cs,ys,fq} for a nice expository and \cite{bcd,bc,ct,t} for the operator defined by the classical spectral
theory and \cite{cscs,dd,stein} for the operator defined via the Riesz potential.

In this paper, our interest in problem \eqref{1.1} is related to the following Hardy inequality which was proved by
Herbst in \cite{h} (see also \cite{bel,y}):
 \begin{equation}\label{1.2}
\bar{\mu}\int_{\R^{N}}\frac{u^2}{|x|^{2s}}dx\leq \int_{\R^{N}}|\xi|^{2s}\hat{u}^2d\xi,
\,\,\,\forall u\in C^{\infty}_0(\R^N),
\end{equation}
where $\hat{u}$ is the Fourier transform of $u$ and
$$
\bar{\mu}=2^{2s}\frac{\Gamma^2(\frac{N+2s}{4})}{\Gamma^2(\frac{N-2s}{4})}.
$$
Here $\Gamma$ is the usual gamma function, the constant $\bar{\mu}$ is optimal and
converges to the classical Hardy constant $\frac{(N-2)^2}{4}$ when $s\rightarrow1$. Indeed, for $\alpha\in[0,\frac{N-2s}{2})$, if we
denote
$$\Upsilon_\alpha=
2^{2s}\frac{\Gamma(\frac{N+2s+2\alpha}{4})\Gamma(\frac{N+2s-2\alpha}{4})}{\Gamma(\frac{N-2s-2\alpha}{4})\Gamma(\frac{N-2s+2\alpha}{4})},
$$
then $\bar{\mu}=\Upsilon_0$, $\Upsilon_\alpha\rightarrow0$ when $\alpha\rightarrow\frac{N-2s}{2}$ and
the mapping $\alpha\longmapsto\Upsilon_\alpha$ is monotone decreasing (see \cite{fa}).

Taking into account the behavior of the Fourier transform with respect to the homogeneity,
one has (see \cite{dp,ls,l,f2})
 \begin{equation*}\label{1.2.1}
\big\|(-\Delta)^{\frac{s}{2}} u\big\|_{L^2(\R^N)}= \int_{\R^{N}}|\xi|^{2s}\hat{u}^2d\xi,
\,\,\,\forall u\in C^{\infty}_0(\R^N),
\end{equation*}
and then the Hardy inequality \eqref{1.2} can be rewritten as
 \begin{equation}\label{1.2.2}
\big\|(-\Delta)^{\frac{s}{2}} u\big\|_{L^2(\R^N)}\geq\bar{\mu}\int_{\R^{N}}\frac{u^2}{|x|^{2s}}dx,
\,\,\,\forall u\in C^{\infty}_0(\R^N).
\end{equation}
In a more general setting, beyond of the Hilbertian framework, we can refer the reader to
\cite{f,f3} where an improved inequality is proved.

Recently, the semilinear fractional elliptic equations involving Hardy potential
 \begin{equation}\label{1.2.3}
(-\Delta)^su-\frac{\mu u}{|x|^{2s}}=f(u),\,\,\,\text{in}\,\Om,
\end{equation}
have been widely studied since the operator $(-\Delta)^s-\mu|x|^{-2s}$ appears in the problem of stability of relativistic matter
in magnetic fields. In \cite{fa}, Fall studied \eqref{1.2.3} with $f(u)=u^p$ and $\Om=B$, and showed that \eqref{1.2.3} possesses a nonnegative distributional solution if $\mu>0$ and $p>1$ satisfying some suitable conditions.
Replacing $f(u)=u^p+\la u^q$, Barrios, Medina and Peral  \cite{bmp} considered \eqref{1.2.3} and discussed the existence
and multiplicity of solutions depending on the value $p$ to \eqref{1.2.3}. Particularly, they verified the existence of \eqref{1.2.3}
if $p=p(\mu,s)=\frac{N+2s-2\alpha_s}{N-2s-2\alpha_s}$ is the threshold for $\alpha_s\in(0,\frac{N-2s}{2})$.
Choi and Seok \cite{cs} considered problem \eqref{1.1} with $\mu=0$. They obtained the existence of
infinity many solutions of \eqref{1.1} for any $a>0$.

In \cite{cs1}, Cao and Yan also considered problem \eqref{1.1}
with $s=1$. It was proved that \eqref{1.1} has infinitely many solutions if $N\geq7$ and $0\leq\mu<\frac{(N-2)^2}{4}-4$
with $a>0$ and $s=1$.
So motivated by \cite{cs,cs1}, the aim of this paper is to study the existence of
infinity many solutions of the Hardy-Sobolev fractional equation \eqref{1.1}. Now we state our main result as follows.

\begin{thm}\label{thm1.1}
Suppose that $a>0$, $N>6s$ and $0\leq\mu<\Upsilon_s$ satisfying $\frac{2^*_s\sqrt{\bar{\mu}}}{\sqrt{\bar{\mu}}-\sqrt{\bar{\mu}-\mu}}>\frac{2N}{N-6s}$.
Then \eqref{1.1} has infinitely many solutions.
\end{thm}

\begin{remark}\label{re1.2}
Our result extends the results in \cite{cs,yyy}
for the particular case of $\mu=0.$  Since there is no Hardy term in \cite{cs},
they only required that $N>6s.$
\end{remark}

\begin{remark}\label{re1.3}
When $s\rightarrow 1,$  the assumptions that $0\leq\mu<\Upsilon_s$ and $\frac{2^*_s\sqrt{\bar{\mu}}}{\sqrt{\bar{\mu}}-\sqrt{\bar{\mu}-\mu}}>\frac{2N}{N-6s}$
in Theorem \ref{thm1.1} are equivalent to
$0\leq \mu<\frac{N(N-4)}{4}$ and $\mu<\frac{(N-2)^{2}}{4}-4$ respectively, that is
just $0\leq \mu<\frac{(N-2)^{2}}{4}-4.$ So our result is uniform with the result in \cite{cs1}
when $s\rightarrow 1.$
\end{remark}

As in \cite{cs1,cs}, one of the main difficulties to prove Theorem
\ref{thm1.1} by using variational methods is that $I(u)$ does not
satisfy the Palais-Smale condition for large energy level, since
$2^{*}_s$ is the critical exponent for the Sobolev embedding from
$H^{s}(\Omega)$ to $L^{2^*_s}(\Omega)$. Another difficulty is that, unlike \cite{cs}, every nontrivial
 solution of \eqref{1.1} is singular at $\{x=0\}$ if $\mu\neq 0$ (see \cite{bmp}). So different techniques are needed to deal with
 the case $\mu\neq 0$. In order to overcome
the first difficulty, we first look at the following perturbed problem:
\begin{equation}\label{1.3}
\left\{%
\begin{array}{ll}
    (-\Delta)^s u-\ds\frac{\mu u}{|x|^{2s}}=|u|^{2^*_s-2-\epsilon}u+a u, & \hbox{$\text{in}~ \Omega$},\vspace{0.1cm} \\
   u=0,\,\, &\hbox{$\text{on}~\partial \Omega$}, \\
\end{array}%
\right.
\end{equation}
where $\epsilon>0$ is small.

The functional corresponding to \eqref{1.3} becomes
\begin{equation}\label{1.4}
I_{\epsilon}(u)=\frac{1}{2}\ds\int_{\Om}\Big(\big|(-\Delta)^{\frac{s}{2}} u\big|^2-\mu\frac{ u^{2}}{|x|^{2s}}-au^2\Big)dx
-\frac{1}{2^*_s-\epsilon}\ds\int_{\Om}|u|^{2^*_s-\epsilon}dx,\,\,\,u\in
H^{s}_{0}(\Omega).
\end{equation}
 Now $I_{\epsilon}(u)$ is an even functional and satisfies the Palais-Smale
 condition in all energy levels. It follows from the symmetric
 mountain-pass lemma \cite{ar,r}, \eqref{1.3} has infinitely many solutions.
  More precisely, there are positive numbers $c_{\epsilon,l},l=1,2,\cdot\cdot\cdot,$ with $c_{\epsilon,l}\rightarrow\infty$
 as $l\rightarrow +\infty$,
 and a solution $u_{\epsilon,l}$ for \eqref{1.3} satisfying
 $$
I_{\epsilon}(u_{\epsilon,l})=c_{\epsilon,l}.
 $$

Moreover, $c_{\epsilon,l}\rightarrow c_{l}<+\infty$ as
$\epsilon\rightarrow 0.$ Now we want to study the behavior of $u_{\epsilon,l}$ as
$\epsilon\rightarrow 0$ for each fixed $l$. If we can prove that $u_{\epsilon,l}$ converges strongly in $H_{0}^{s}(\Omega)$ to
$u_{l}$ as $\epsilon\rightarrow0$, then $u_l$ is a solution of \eqref{1.1} with $I(u_l)=c_l$. This will imply that \eqref{1.1} has
infinitely many solutions. Thus
Theorem \ref{thm1.1} is a direct consequence of the following
result.
\begin{thm}\label{thm1.2}
Suppose that $a>0$, $N>6s$ and $0\leq\mu<\Upsilon_s$ satisfying $\frac{2^*_s\sqrt{\bar{\mu}}}{\sqrt{\bar{\mu}}-\sqrt{\bar{\mu}-\mu}}>\frac{2N}{N-6s}$.
Then for any sequence $u_{n},$ which is
a solution of \eqref{1.3} with $\epsilon=\epsilon_{n}\rightarrow 0,$
satisfying $\|u_{n}\|\leq C$ for some constant independent of $n$,
$u_{n}$ has a subsequence, which converges strongly in
$H^{s}_{0}(\Omega)$ as $n\rightarrow\infty.$
\end{thm}

Following the ideas in \cite{cps,cs1,cs}, to prove Theorem \ref{thm1.2}, we shall prove the strong convergence of $u_{\epsilon,l}$
by using a local Pohozaev identity to exclude the possibility of concentration.
We would like to point out that since an important feature of the fractional Laplacian is its nonlocal property,
it turns out from several technical reasons that studying our nonlocal equation \eqref{1.3} directly is not suitable for establishing Theorem \ref{thm1.2}.
So different from \cite{cps,cs1}, like \cite{cs,yyy}, we need to
realize the nonlocal operator $(-\Delta)^s$ in $\Om$ through a local problem in $\Om\times(0,\infty)$.

To explain this, we have to introduce some more function spaces on
$\mathcal{D}=\Om\times(0,\infty)$, where $\Om$ is either a smooth bounded domain or $\R^N$.
If $\Om$ is bounded, then we define the function space $H^1_0(t^{1-2s},\mathcal{D})$
as the completion of
$$
C_{0,L}^{\infty}(\mathcal{D}):=\{\bar{u}\in C^{\infty}(\bar{\mathcal{D}}): \bar{u}=0\,
\text{on}\,\partial_L\mathcal{D}=\partial\Om\times[0,\infty)\}
$$
with respect to the norm
\begin{equation}\label{1.7}
\|\bar{u}\|_{H^1_0(t^{1-2s},\mathcal{D})}
=\Big(\ds\int_{\mathcal{D}}t^{1-2s}|\nabla \bar{u}|^2dxdt\Big)^{\frac{1}{2}}.
\end{equation}
Then it is a Hilbert space endowed with the inner product
$$
(\bar{u},\bar{v})_{H^1_0(t^{1-2s},\mathcal{D})}
=\ds\int_{\mathcal{D}}t^{1-2s}\nabla \bar{u}\nabla \bar{v}dxdt.
$$
In the same manner, we define the space $D^1(t^{1-2s},\R_+^{N+1})$ as the completion of
$C_0^{\infty}(\overline{\R_+^{N+1}})$ with respect to the norm
$$
\|\bar{u}\|_{D^1(t^{1-2s},\R_+^{N+1})}
=\Big(\ds\int_{\R_+^{N+1}}t^{1-2s}|\nabla \bar{u}|^2dxdt\Big)^{\frac{1}{2}}.
$$
Recall that if $\Om$ is a smooth bounded domain, then
it is verified that (see  \cite{cscs}, Proposition 2.1; \cite{cdds}, Proposition 2.1; \cite{t1}, Section 2)
$$H_0^s(\Om)=\{u=tr|_{\Om\times\{0\}}\bar{u}:\bar{u}\in H^1_0(t^{1-2s},\mathcal{D})\},$$
and
\begin{equation}\label{1.8}
\|\bar{u}(\cdot,0)\|_{H^s_0(\Om)}
\leq C\|\bar{u}\|_{H^1_0(t^{1-2s},\mathcal{D})}
\end{equation}
for some $C>0$ independent of $\bar{u}\in H^1_0(t^{1-2s},\mathcal{D})$. Similarly, it holds by taking trace that
$$D^s(\R^N)=\{u=tr|_{\R^N\times\{0\}}\bar{u}:\bar{u}\in D^1(t^{1-2s},\R_+^{N+1})\},$$
and
\begin{equation}\label{1.9}
\|\bar{u}(\cdot,0)\|_{D^s(\R^N)}
\leq C\|\bar{u}\|_{D^1(t^{1-2s},\R_+^{N+1})}
\end{equation}
for some $C>0$ independent of $\bar{u}\in D^1(t^{1-2s},\R_+^{N+1})$.

Now we are ready to consider the fractional harmonic extension of a function $u$ defined in $\Om$,
where $\Om$ is either a smooth bounded domain or $\R^N$. By the celebrated results of Caffarelli and
Silvestre \cite{cscs} (for $\R^N$) and Cabr\'{e} and Tan \cite{ct} (for bounded domains, see also \cite{bc,s,t1}), if we set $\bar{u}\in H^1_0(t^{1-2s},\mathcal{D})$ (or $D^1(t^{1-2s},\R_+^{N+1})$) as the unique solution of the equation
\begin{equation}\label{1.10}
\left\{%
\begin{array}{ll}
    div(t^{1-2s}\nabla \bar{u})=0, & \hbox{$\text{in}~ \mathcal{D}$},\vspace{0.1cm} \\
   \bar{u}=0,\,\, &\hbox{$\text{on}~\partial_L \mathcal{D}$},\vspace{0.1cm} \\
   \bar{u}(x,0)=u(x)&\hbox{$\text{for}~x\in \Om$} \\
\end{array}%
\right.
\end{equation}
for some fixed function $u\in H^s_0(\Om)$ (or $D^s(\R^N)$), then
$$\mathcal{A}_s\bar{u}:=-d_s\lim\limits_{t\rightarrow0^+}t^{1-2s}\frac{\partial \bar{u}}{\partial t}(x,t)\,\,\,\text{for}\,x\in \Om$$
is well defined and one must have
$$
(-\Delta)^su=\mathcal{A}_s \bar{u},
$$
with $$d_s:=\frac{2^{1-2s}\Gamma(1-s)}{\Gamma(s)}.$$
Without loss of generality, we may assume throughout this paper that $d_s=1$, that is,
\begin{equation}\label{1.11}
(-\Delta)^su=\mathcal{A}_s \bar{u}=-\lim\limits_{t\rightarrow0^+}t^{1-2s}\frac{\partial \bar{u}}{\partial t}(x,t).
\end{equation}
We call this $\bar{u}$ the $s$-harmonic extension of $u$ and we point out that by a density argument, \eqref{1.11} is satisfied
in weak sense for $u\in H^s_0(\Om)$ (or $D^s(\R^N)$). In other words, for any $u,\phi \in H^s_0(\Om)$ (or $D^s(\R^N)$), there holds
$$\left \langle u,\phi\right\rangle_{H_0^s(\Om)}=\left \langle \bar{u},\bar{\phi}\right\rangle_{H^1_0(t^{1-2s},\mathcal{D})}.$$
Thus the trace inequality \eqref{1.8} is improved as
\begin{equation}\label{1.12}
\|\bar{u}(\cdot,0)\|_{H^s_0(\Om)}
=\|\bar{u}\|_{H^1_0(t^{1-2s},\mathcal{D})}.
\end{equation}

Therefore from the above analysis, we can deduce that if a function $u$ is a weak solution to the nonlocal problem
\begin{equation}\label{1.13}
\left\{%
\begin{array}{ll}
    (-\Delta)^s u=g, & \hbox{$\text{in}~ \Omega$},\vspace{0.1cm} \\
   u=0,\,\, &\hbox{$\text{on}~\partial \Omega$}, \\
\end{array}%
\right.
\end{equation}
if and only if its $s$-harmonic extension $\bar{u}$ is a weak solution to the local problem
\begin{equation}\label{1.14}
\left\{%
\begin{array}{ll}
    div(t^{1-2s}\nabla \bar{u})=0, & \hbox{$\text{in}~ \mathcal{D}$},\vspace{0.1cm} \\
   \bar{u}=0,\,\, &\hbox{$\text{on}~\partial_L \mathcal{D}$},\vspace{0.1cm} \\
   \mathcal{A}_s\bar{u}=g(x),&\hbox{$\text{on}~ \Om \times \{0\}$}, \\
\end{array}%
\right.
\end{equation}
where $g\in L^{\frac{2N}{N+2s}}(\Om)$. Moreover, we say that a function $u\in H^s_0(\Om)$ is a weak solution of
\eqref{1.13} provided
\begin{equation}\label{1.15}
  \ds\int_{\Om}(-\Delta)^{\frac{s}{2}} u(-\Delta)^{\frac{s}{2}}\phi dx=\ds\int_{\Om}g\phi dx
  \end{equation}
for all $\phi\in H^s_0(\Om)$ and a function $\bar{u}$ is a weak solution of \eqref{1.14} if
\begin{equation}\label{1.16}
\ds\int_{\mathcal{D}}t^{1-2s}\nabla \bar{u}\nabla \bar{\phi}dxdt=\ds\int_{\Om}g\bar{\phi}(x,0)dx
\end{equation}
for all $\bar{\phi}\in H^1_0(t^{1-2s},\mathcal{D})$.

Hence, as mentioned before, rather than studying the nonlocal problem \eqref{1.1} directly, it is better to consider
the so-called $s$-harmonic extension problem
\begin{equation}\label{1.17}
\left\{%
\begin{array}{ll}
    div(t^{1-2s}\nabla \bar{u})=0, & \hbox{$\text{in}~ \mathcal{D}$},\vspace{0.1cm} \\
   \bar{u}=0,\,\, &\hbox{$\text{on}~\partial_L \mathcal{D}$},\vspace{0.1cm} \\
   \mathcal{A}_s\bar{u}=\ds\frac{\mu u}{|x|^{2s}}+|u|^{2^*_s-2}u+a u,&\hbox{$\text{on}~ \Om \times \{0\}$}. \\
\end{array}%
\right.
\end{equation}
By virtue of considering \eqref{1.17}, one can easily obtain the decomposition of
approximating solutions  and establish a local Pohozaev identity in small balls which may contain the origin.
Then applying this identity, we can exclude the possibility of concentration and prove the strong convergence of
approximating solutions.

Theorems \ref{thm1.1} and \ref{thm1.2}
extend the results in \cite{cs,yyy} to the fractional Laplacian
problem with Hardy term. We want to stress that
it is more difficult to obtain the estimates
in order to prove these results for \eqref{1.1}.
Like \cite{cs,yyy}, the main difficulty in the study
of \eqref{1.17} is that we need to carry out the boundary
estimates. This is greatly different
from the usual Laplacian equations studied in \cite{cps,cs1,cs}
which mainly involving the interior estimates.

Throughout this paper, we denote the norm of  $H^{s}_{0}(\Omega)$ by
$\|u\|=\Bigl(\ds\int_{\Omega}\big|(-\Delta)^{\frac{s}{2}} u\big|^2dx\Bigl)^{\frac{1}{2}}$;
the norm of $L^{q}(\Omega)(1\leq q<\infty)$ by
$\|u\|_q=\Bigl(\ds\int_{\Omega}|u|^qdx\Bigl)^{\frac{1}{q}}$;
the norm of $L^{q}(t^{1-2s},\Omega)(1\leq q<\infty)$ by
$\|u\|_{L^q(t^{1-2s},\Omega)}=\Bigl(\ds\int_{\Omega}t^{1-2s}|u|^qdxdt\Bigl)^{\frac{1}{q}}$; moreover, we denote
$B_r^N(x):=\big\{y\in\R^N:\,|y-x|\leq r\big\},$
$B_r^{N+1}(x):=\big\{z=(y,t)\in\R^{N+1}_+:\,|z-(x,0)|\leq r\big\},$
and positive constants (possibly different) by $C$. For simplicity, sometimes we also write
$B_r^N(x)$ as $B_r(x)$.

 The organization of our paper is as follows.
  In Section 2, we will give some integral estimates.
  In Section 3, we obtain some estimates on safe regions. We will  prove our main result in Section 4.
  In order that we can give a clear line of our framework, we will  list  some estimates for linear problems with
  Hardy potential, an iteration result, a decay estimate, a local pohozaev identity and the decomposition of
  approximating solutions in Appendices A, B, C and D.

\section{Some Integral estimates}

For any $\Lambda>0$ and $x\in \mathbb{R}^N$, we define
\[
\rho_{x,\Lambda}(u)=\Lambda^{\frac{N}{2^{*}_s}}u\bigl(\Lambda
(\cdot-x)\bigr),\,\,\,\,\,u\in H^{s}_{0}(\Omega).
\]

Let $u_{n}$ be a solution of \eqref{1.3} with $\epsilon=\epsilon_{n}\rightarrow 0$,
satisfying $\|u_{n}\|\leq C$ for some constant $C$ independent of $n$, by Proposition
\ref{propc.1}, $u_{n}$ can be decomposed as
$$
u_{n}=u_{0}+\sum_{j=1}^{m}\rho_{0,\Lambda_{n,j}}(U_{j})
+\sum_{j=m+1}^{h}\rho_{x_{n,j},\Lambda_{n,j}}(U_{j})+\omega_{n}.
$$
In this section, we will prove a Brezis-Kato type estimate (see
\cite{bk}).

In order to prove the strong convergence of $u_{n}$ in
$H^{s}_{0}(\Omega)$, we only need to show that the bubbles
$\rho_{x_{n,j},\Lambda_{n,j}}(U_{j})$ will not appear in the
decomposition of $u_{n}$.

Among all the bubbles $\rho_{x_{n,j},\Lambda_{n,j}}(U_{j})$, we can
choose a bubble, such that this bubble has the slowest concentration
rate. That is, there is $j_{0}$ such that the corresponding $\Lambda_{n,j_{0}}$ is the lowest order
infinity among all the $\Lambda_{n,j}$ appearing in the bubbles. For
simplicity, we denote $\Lambda_{n}$ the slowest concentration rate
and $x_{n}$ the corresponding concentration point.

\begin{rem}\label{remh2.1}
Since $\ds\frac{1}{|x|^{2s}}\in C^{2s}(\Omega\setminus B_{\delta}(0))$ for any $\delta>0$ small,
we know that $u_{n},u_{0},U_{j}\in C^{2s}(\Omega\setminus B_{\delta}(0))$.
As a result, $\omega_{n}\in C^{2s}(\Omega\setminus B_{\delta}(0))$ for any $\delta>0$ small.
\end{rem}

For any $p_{2}<2^{*}_s<p_{1},\alpha>0$ and $\Lambda\geq0$, we
consider the following relation:
\begin{equation}\label{2.1}
\left\{%
\begin{array}{ll}
   ||u_{1}||_{p_{1}}\leq \alpha,\vspace{0.1cm} \\
  ||u_{2}||_{p_{2}}\leq \alpha\Lambda^{\frac{N}{2^{*}_s}-\frac{N}{p_{2}}}. \\
\end{array}%
\right.
\end{equation}
Define
$$
\|u\|_{p_{1},p_{2},\Lambda}=\inf\bigl\{\alpha>0:\text{there~are}\,\,u_{1}\,\,\text{and}\,\,u_{2},\,\,\text{such~that}\,\,
\eqref{2.1} \,\,\text{holds~and}\,\,|u|\leq u_{1}+u_{2}\bigl\}.
$$

To deal with the Hardy potential, we need another norm. Consider the following relation:
\begin{equation}\label{mz}
\left\{%
\begin{array}{ll}
  | |u_{1}||_{*,p_{1}}\leq \alpha,\vspace{0.1cm} \\
  ||u_{2}||_{*,p_{2}}\leq \alpha\Lambda^{\frac{N}{2^{*}_s}-\frac{N}{p_{2}}}, \\
\end{array}%
\right.
\end{equation}
where
$$
 ||u||_{*,p}= ||u||_{p}+\Big(\mu\int_{\Omega}\frac{|u|^{\frac{2p}{2^{*}_s}}}{|x|^{2s}}dx\Big)^{\frac{2^{*}_s}{2p}}.
$$

Define
$$
\|u\|_{*,p_{1},p_{2},\Lambda}=\inf\bigl\{\alpha>0:\text{there~are}\,\,u_{1}\,\,\text{and}\,\,u_{2},\,\,\text{such~that}\,\,
\eqref{mz} \,\,\text{holds~and}\,\,|u|\leq u_{1}+u_{2}\bigl\}.
$$

From the definitions, it is easy to see that $\|u\|_{p_{1},p_{2},\Lambda}\leq \|u\|_{*,p_{1},p_{2},\Lambda}.$

Let $w_{n}(x)=|u_{n}(x)|$ in $\Omega$; $w_{n}(x)=0$ in $\R^{N}\backslash\Omega$. Then it is easy to check that $w_{n}$ satisfies the following inequality
\begin{equation}\label{h2.4}
    \int_{\R^{N+1}_{+}}t^{1-2s}\nabla \bar{w}_{n}\nabla \bar{\phi}dxdt-\mu\int_{\R^{N}}\frac{w_{n}\phi}{|x|^{2s}}dx\leq
    \int_{\R^{N}}\big(2w_{n}^{2^{*}_s-1}+A\big)\phi dx ,\forall \bar{\phi}\in H^{1}(t^{1-2s},\R^{N+1}_{+}),\bar{\phi}\geq 0,
\end{equation}
where $A>0$ is a large constant.

The main result of this section is the following proposition.

\begin{prop}\label{prop2.1}
Let $w_{n}$ be a solution of \eqref{h2.4}. For any $p_{1},p_{2}\in \big(\frac{2^{*}_s}{2},\frac{2^*_s\sqrt{\bar{\mu}}}{\sqrt{\bar{\mu}}-\sqrt{\bar{\mu}-\mu}}\big)$
satisfying $p_{2}<2^{*}_s<p_{1}$, there is a
constant $C$, depending on $p_{1}$ and $p_{2}$, such that
$$
\|w_{n}\|_{*,p_{1},p_{2},\Lambda_{n}}\leq C.
$$
\end{prop}

To prove Proposition \ref{prop2.1}, we should show the following three lemmas.

\begin{lem}\label{lem2.3}
Let $w$ be the solution of
$$
\left\{%
\begin{array}{ll}
    (-\Delta)^s w-\ds\frac{\mu w}{|x|^{2s}}=a(x)v, & \hbox{$\text{in}~ \Omega$},\vspace{0.1cm} \\
   w=0,\,\, &\hbox{$\text{on}~\partial \Omega$}, \\
\end{array}%
\right.
$$
where $a(x)\geq 0$ and $v\geq 0$ are functions satisfying $a,v\in
C^{2s}(\Omega\setminus B_{\delta}(0))$ for any $\delta>0.$ Then for any
$\frac{N}{N-2s}<p_{2}<2^{*}_s<p_{1}$ and $0\leq \mu<\bar{\mu}$ satisfying $p_{1}<\frac{2^*_s\sqrt{\bar{\mu}}}{\sqrt{\bar{\mu}}-\sqrt{\bar{\mu}-\mu}}$, there is a constant
$C=C(p_{1},p_{2})$ such that for any $\Lambda\geq 1$,
$$
\|w\|_{*,p_{1},p_{2},\Lambda}\leq
C||a||_{\frac{N}{2s}}||v||_{p_{1},p_{2},\Lambda}.
$$
\end{lem}

\begin{proof}
For any small $\theta>0,$ let $v_{1},v_{2}\in
C^{2s}(\Omega\setminus B_{\delta}(0))$ and $v_{1},v_{2}\geq 0$ be the
functions such that $v\leq v_{1}+v_{2} $ and \eqref{2.1} holds with
$\alpha=\|v\|_{p_{1},p_{2},\Lambda}+\theta$. Consider
$$
\left\{%
\begin{array}{ll}
    (-\Delta)^s w_{i}-\ds\frac{\mu w_{i}}{|x|^{2s}}=a(x)v_{i}, & \hbox{$\text{in}~ \Omega$},\vspace{0.1cm} \\
   w_{i}=0,\,\, &\hbox{$\text{on}~\partial \Omega$}. \\
\end{array}%
\right.
$$

It follows from Lemma \ref{lema.1} that
\begin{equation}\label{w}
||w_{i}||_{*,p_{i}}\leq
C||a||_{\frac{N}{2s}}||v_{i}||_{p_{i}},\,\,i=1,2.
\end{equation}

On the other hand, by the maximum  principle, we deduce $w\le w_1+w_2$. So the result follows.
\end{proof}

\begin{rem}\label{remh2.4}
We will let $a(x)=|u_{0}|^{\frac{4s}{N-2s}}$ or $a(x)=|\rho_{0,\Lambda_{n,j}}|^{\frac{4s}{N-2s}},j=1,2,\cdots,m,$ in Lemma \ref{lem2.3} to obtain some desired estimates for $w_{n}$. Here $a(x)$ may have singularity at $\{x=0\}$. So, in Lemma \ref{lem2.3} we only assume that $a(x)$ and $v(x)$ belong to $C^{2s}(\Omega\setminus B_{\delta}(0)).$
\end{rem}

\begin{lem}\label{lem2.4}
Let $w\geq 0$ be a weak solution of
$$
\left\{%
\begin{array}{ll}
    (-\Delta)^s w-\ds\frac{\mu w}{|x|^{2s}} =2v^{2^{*}_s-1}+A, & \hbox{$\text{in}~ \Omega$},\vspace{0.1cm} \\
   w=0,\,\, &\hbox{$\text{on}~\partial \Omega$}. \\
\end{array}%
\right.
$$
For any
$p_{1},p_{2}\in\big(\frac{N+2s}{N-2s},\frac{N}{2s}\frac{N+2s}{N-2s}\big)$ with
$p_{2}<2^{*}_s<p_{1},$ let $q_{i}$ be given by
$$
\frac{1}{q_{i}}=\frac{N+2s}{(N-2s)p_{i}}-\frac{2s}{N},i=1,2.
$$
If $q_{1},q_{2}<\frac{2^*_s\sqrt{\bar{\mu}}}{\sqrt{\bar{\mu}}-\sqrt{\bar{\mu}-\mu}}$, then there is a constant $C=C(p_{1},p_{2})$ such that for any
$\Lambda\geq 1,$
$$
\|w\|_{*,q_{1},q_{2},\Lambda}\leq
C\|v\|^{2^{*}_s-1}_{p_{1},p_{2},\Lambda}+C.
$$
\end{lem}

\begin{proof}
For any small $\theta>0$, let $v_{1}\geq 0$ and $v_{2}\geq 0$ be the
functions such that $v\leq v_{1}+v_{2}, $ and \eqref{2.1} holds with
$\alpha=\|v\|_{p_{1},p_{2},\Lambda}+\theta$.

Consider
$$
\left\{%
\begin{array}{ll}
    (-\Delta)^s w_{1}-\ds\frac{\mu w_{1}}{|x|^{2s}}=Cv_{1}^{2^{*}_s-1}+A,
    & \hbox{$\text{in}~ \Omega$},\vspace{0.1cm} \\
   w_{1}=0,\,\, &\hbox{$\text{on}~\partial \Omega,$} \\
\end{array}%
\right.
$$
and
$$
\left\{%
\begin{array}{ll}
    (-\Delta)^s w_{2}-\ds\frac{\mu w_{2} }{|x|^{2s}}=Cv_{2}^{2^{*}_s-1},
    & \hbox{$\text{in}~ \Omega$},\vspace{0.1cm} \\
   w_{2}=0,\,\, &\hbox{$\text{on}~\partial \Omega$},\\
\end{array}%
\right.
$$
where $C>0$ is a large constant. Then by the maximum principle, $w\le w_1+w_2$.

Let $\hat{p}_{i}=p_{i}\frac{N-2s}{N+2s}\in(1,\frac{N}{2s}),$ then $q_{i}=\frac{N\hat{p}_{i}}{N-2s\hat{p}_{i}},i=1,2.$
 By Lemma \ref{lema.2}, we have
 \begin{eqnarray*}
||w_{1}||_{*,q_{1}}
&\leq&C\big\|v_{1}^{2^{*}_s-1}+A\big\|_{\hat{p}_{1}}
\leq C\big(||v_{1}||^{2^{*}_s-1}_{p_{1}}+1\big)
\leq C\bigl[(\|v\|_{p_{1},p_{2},\Lambda}+\theta)^{2^{*}_s-1}+1\bigl],
\end{eqnarray*}
and
\begin{eqnarray*}
||w_{2}||_{*,q_{2}}&\leq&
C||v_{2}||^{2^{*}_s-1}_{p_{2}}
\leq
C\Bigl[(\|v\|_{p_{1},p_{2},\Lambda}+\theta)\Lambda^{\frac{N}{2^{*}_s}-\frac{N}{p_{2}}}\Bigl]^{2^{*}_s-1}
=C(\|v\|_{p_{1},p_{2},\Lambda}+\theta)^{2^{*}_s-1}\Lambda^{\frac{N}{2^{*}_s}-\frac{N}{q_{2}}},
\end{eqnarray*}
since
$$
\Bigl(\frac{N}{2^{*}_s}-\frac{N}{p_{2}}\Bigl)\frac{N+2s}{N-2s}=\frac{N}{2^{*}_s}-\frac{N}{q_{2}}.
$$
As a consequence, the result follows.
\end{proof}

\begin{lem}\label{lem2.5}
Let $w_{n}(x)=|u_{n}(x)|$ in $\Omega$; $w_{n}(x)=0$ in $\R^{N}\setminus\Omega$. Then there are
constants $C>0$ and $p_{1},p_{2}\in\big(\frac{2^{*}_s}{2},+\infty\big)$
with $p_{2}<2^{*}_s<p_{1}$ such that
$$
\|w_n\|_{*,p_{1},p_{2},\Lambda_{n}}\leq C.
$$
\end{lem}

\begin{proof}
From Proposition \ref{propc.1}, we have
$u_{n}=u_{0}+u_{n,1}+u_{n,2},$ where
$$
u_{n,1}=\sum_{j=1}^{m}\rho_{0,\Lambda_{n,j}}(U_{j})+\sum_{j=m+1}^{h}\rho_{x_{n,j},\Lambda_{n,j}}(U_{j}),
$$
and $u_{n,2}=\omega_{n}.$

Let $a_{i}=C|u_{n,i}|^{\frac{4s}{N-2s}},i=0,1,2,$ where $C>0$ is a large constant.
Then, we have
$$
(-\Delta)^s w_{n}-\frac{\mu w_{n}}{|x|^{2s}}\leq(a_{0}+a_{1}+a_{2})w_{n}+A.
$$

Let $w=G(v)$ be the solution of
$$
\left\{%
\begin{array}{ll}
    (-\Delta)^s w-\ds\frac{\mu w}{|x|^{2s}}=v,
    & \hbox{$\text{in}~ \Omega$},\vspace{0.1cm} \\
   w=0,\,\, &\hbox{$\text{on}~\partial \Omega$}. \\
\end{array}%
\right.
$$
Then, we have
\begin{eqnarray*}
w_{n}\leq
G(a_{0}w_{n}+A)+G(a_{1}w_{n})+G(a_{2}w_{n}).
\end{eqnarray*}

We first deal with the term $G(a_{0}w_{n}+A).$ Let $q>\frac{2N}{N+2s}$ such that
$q-\frac{2N}{N+2s}$ is so small that
$p_{1}:=\frac{Nq}{N-2sq}\in \Big(2^{*}_s,\frac{2^*_s\sqrt{\bar{\mu}}}{\sqrt{\bar{\mu}}-\sqrt{\bar{\mu}-\mu}}\Big).$
Then it follows from Lemma \ref{lema.2} that
\begin{equation}\label{*2.7}
\begin{array}{ll}
\big\|G(a_{0}w_{n}+A)\big\|_{*,p_{1}}&\leq C||a_{0}w_{n}+A||_{q}
\leq C+C\|a_{0}\|_{\frac{2^{*}_sq}{2^{*}_s-q}}\|w_{n}\|_{2^{*}_s}
\leq C+C\|a_{0}\|_{\frac{2^{*}_sq}{2^{*}_s-q}}.
\end{array}%
\end{equation}

By Proposition \ref{propbb3}, $|u_{0}(x)|\leq C|x|^{-(\frac{N-2s}{2}-\beta)}, \forall x\in \Om$. So we see that
if $q-\frac{2N}{N+2s}>0$ is small enough, then
$$
\Big(\frac{N-2s}{2}-\beta\Big)\frac{4s}{N-2s}\frac{2^{*}_sq}{2^{*}_s-q}<N.
$$
As a result,
$$
\ds\int_{\Omega}|a_{0}|^{\frac{2^{*}_sq}{2^{*}_s-q}}dx
\leq C\ds\int_{\Omega}|x|^{-\big(\frac{N-2s}{2}-\beta\big)\frac{4s}{N-2s}\frac{2^{*}_sq}{2^{*}_s-q}}dx
\leq C.
$$
Hence, we  have proved that there is a $p_{1}>2^{*}_s$ such that
\begin{equation}\label{h2.8}
\big\|G(a_{0}w_{n}+A)\big\|_{*,p_{1}}\leq C.
\end{equation}

Next we treat the term
$G(a_{1}w_{n}).$ Let
$p_{2}\in(\frac{N}{N-2s},2^{*}_s)$ be a constant. By Lemma
\ref{lema.3}, we get
\begin{eqnarray*}
\big\|G(a_{1}w_{n})\big\|_{*,p_{2}}&\leq&
C||a_{1}||_{r}||w_{n}||_{2^{*}_s}\leq C||a_{1}||_{r},
\end{eqnarray*}
where $r$ is determined by
$\frac{1}{p_{2}}=\frac{1}{r}+\frac{1}{2^{*}_s}-\frac{2s}{N}.$

But
$$
\int_{\Omega}|\rho_{x_{n,j},\Lambda_{n,j}}(U_{j})|^{\frac{4sr}{N-2s}}dx
=\Lambda_{n,j}^{2sr-N}\int_{\Omega_{x_{n,j},\Lambda_{n,j}}}
|U_{j}|^{\frac{4sr}{N-2s}}dx,
$$
where
$\Omega_{x,\Lambda}=\{\bar{x}:x_{n,j}+\Lambda^{-1}\bar{x}\in\Omega\}.$

For $j=m+1,\cdot\cdot\cdot,h,$ there
is a $C>0$ such that (see \cite{cs})
$$
|U_{j}(x)|\leq\frac{ C}{1+|x|^{N-2s}}.
$$

Therefore, for any $r\in\big(\frac{N}{4s},\frac{N}{2s}\big),$ we have
$$
\int_{\Omega_{x_{n,j},\Lambda_{n,j}}}|U_{j}|^{\frac{4sr}{N-2s}}dx\leq
C,\,\,\,j=m+1,\cdot\cdot\cdot,h.
$$

For $j=1,2,\cdots,m,$ by Lemma \ref{lema.1-1} we have
$$
U_{j}\in L^{p}_{loc}(\R^{N}),\forall p<\frac{2^*_s\sqrt{\bar{\mu}}}{\sqrt{\bar{\mu}}-\sqrt{\bar{\mu}-\mu}},
$$
and by Lemma \ref{lemb.1},
$$
|U_{j}(x)|\leq \frac{C}{|x|^{\frac{N-2s}{2}+\beta}},\,\,\,\forall|x|\geq 1.
$$
Note that $r\rightarrow \frac{N}{2s}$ as $p_{2}\rightarrow 2^{*}_s$. Now we choose $p_2$ close to $2^*_s$ so that
$$
\frac{4sr}{N-2s}\Big(\frac{N-2s}{2}+\beta\Big)>N,
$$
and
$$
\frac{4sr}{N-2s}<\frac{2^*_s\sqrt{\bar{\mu}}}{\sqrt{\bar{\mu}}-\sqrt{\bar{\mu}-\mu}}.
$$
As a result,
$$
\int_{\Omega_{x_{n,j},\Lambda_{n,j}}}|U_{j}|^{\frac{4sr}{N-2s}}dx\leq
C,\,\,\,j=1,2,\cdot\cdot\cdot,m.
$$
 Thus we have proved that there is
a $p_{2}<2^{*}_s$ close to $2^{*}_s$ such that
\begin{equation}\label{2.9}
\big\|G(a_{1}w_{n})\big\|_{*,p_{2}}
\leq
C\Lambda_{n}^{2s-\frac{N}{r}}=C\Lambda_{n}^{\frac{N}{2^{*}_s}-\frac{N}{p_{2}}}.
\end{equation}

Finally, we treat the term $G(a_{2}w_{n})$. It
follows from Lemma \ref{lem2.3} that
\begin{equation}\label{2.11}
\begin{array}{ll}
\|G(a_{2}w_{n})\|_{*,p_{1},p_{2},\Lambda_{n}}&\leq
C||a_{2}||_{\frac{N}{2s}}\|w_{n}\|_{p_{1},p_{2},\Lambda_{n}}
\leq \frac{1}{2}\|w_{n}\|_{*,p_{1},p_{2},\Lambda_{n}},
\end{array}
\end{equation}
since
$||a_{2}||_{\frac{N}{2s}}=||\omega_{n}||^{\frac{4s}{N-2s}}_{2^{*}_s}\rightarrow
0$ as $n\rightarrow\infty.$

From \eqref{h2.8}, \eqref{2.9} and \eqref{2.11}, we obtain
\begin{eqnarray*}
&&\|w_{n}\|_{*,p_{1},p_{2},\Lambda_{n}}\\&\leq&\|G(a_{0}w_{n}+A)\|_{*,p_{1},p_{2},\Lambda_{n}}
+\|G(a_{1}w_{n})\|_{*,p_{1},p_{2},\Lambda_{n}}
+\|G(a_{2}w_{n})\|_{*,p_{1},p_{2},\Lambda_{n}}\\
&\leq&\big\|G(a_{0}w_{n}+A)\big\|_{*,p_{1}}
+\bigl\|G(a_{1}w_{n})\bigl\|_{*,p_{2}}\Lambda_{n}^{\frac{N}{p_{2}}-\frac{N}{2^{*}_s}}
+\|G(a_{2}w_{n})\|_{*,p_{1},p_{2},\Lambda_{n}}
\\
&\leq&C+\frac{1}{2}\|w_{n}\|_{*,p_{1},p_{2},\Lambda_{n}}.
\end{eqnarray*}
Hence the result follows.
\end{proof}

\begin{proof}[\textbf{Proof of Proposition \ref{prop2.1}}]
 Since $w_{n}$ satisfies \eqref{h2.4},
we can use Lemmas \ref{lem2.4} and \ref{lem2.5} to prove that
$$
\|w_{n}\|_{*,p_{1},p_{2},\Lambda_{n}}\leq C
$$
holds for any $p_{1},p_{2}$ with
$p_{1},p_{2}\in \big(\frac{2^{*}_s}{2},\frac{2^*_s\sqrt{\bar{\mu}}}{\sqrt{\bar{\mu}}-\sqrt{\bar{\mu}-\mu}}\big)$
satisfying $p_{2}<2^{*}_s<p_{1}$.
\end{proof}

\section{ Estimates on safe regions}

%In this section, we establish some sharp $L^q$ estimates for solution $\{u_n\}$ on some suitable annuli around the slowest bubbling point
%$x_n$, which play a fundamental role to prove our main results.

Since the number of the bubbles of $u_{n}$ is finite, by Proposition
\ref{propc.1} we can always find a constant $\bar{C}>0,$ independent
of $n$, such that the region
$$
\Bigl(B^N_{(\bar{C}+5)\Lambda_{n}^{-\frac{1}{2}}}(x_{n})\backslash
B^N_{\bar{C}\Lambda_{n}^{-\frac{1}{2}}}(x_{n})\Bigl)\cap\,\Omega,
$$
does not contain any concentration point of $u_{n}$ for any $n$. We
call this region a safe region for $u_{n}.$

For $d=N,N+1$, let
$$
\mathcal
{A}^{d}_{n,1}:=\Bigl(B^d_{(\bar{C}+5)\Lambda_{n}^{-\frac{1}{2}}}(x_{n})\backslash
B^d_{(\bar{C}+1)\Lambda_{n}^{-\frac{1}{2}}}(x_{n})\Bigl)\cap\,\Omega\,(\text{or}\,\mathcal{D}),
$$
and
$$
\mathcal
{A}^{d}_{n,2}:=\Bigl(B^d_{(\bar{C}+4)\Lambda_{n}^{-\frac{1}{2}}}(x_{n})\backslash
B^d_{(\bar{C}+1)\Lambda_{n}^{-\frac{1}{2}}}(x_{n})\Bigl)\cap\,\Omega\,(\text{or}\,\mathcal{D}).
$$

For a measurable set $E\subset \R^{N+1}_+$, we define a weighted measure
$$m_s(E):=\ds\int_{E}t^{1-2s}dxdt,$$
and
$$-\!\!\!\!\!\!\int_{E}t^{1-2s} f(x,t)dxdt:=\frac{\ds\int_{E}t^{1-2s}f(x,t)dxdt}{m_s(E)}.$$

Firstly, we introduce the following two known results given in \cite{fk} and \cite{cs} respectively.

\begin{lem}\label{lem3.1.1}
(Theorem 1.3,  \cite{fk}) Let $\mathcal{F}$ be an open bounded set in $\R^{N+1}$. Then there exists a constant $C(N,s,\mathcal{F})>0$
such that
\begin{equation}\label{3.1.1.1}
\arraycolsep=1.5pt
\begin{array}{rl}\displaystyle
\Big(\ds\int_{\mathcal{F}}t^{1-2s}|\bar{u}(x,t)|^{\frac{2(N+1)}{N}}dxdt\Big)^{\frac{N}{2(N+1)}}
\leq C\Big(\ds\int_{\mathcal{F}}t^{1-2s}|\nabla\bar{u}(x,t)|^2dxdt\Big)^{\frac{1}{2}}.
\end{array}
\end{equation}
\end{lem}

\begin{lem}\label{lem3.1.2}
(Lemma 5.2,  \cite{cs})
For $f\geq0$, assume that $\bar{u}\in H^1_0(t^{1-2s},\mathcal{D})$ satisfies
$$
\left\{%
\begin{array}{ll}
    div(t^{1-2s}\nabla \bar{u})=0,  &\text{in}\;\mathcal{D},\vspace{0.1cm}\\
\mathcal{A}_s(\bar{u})=f, &\text{on}\; \Om\times \{0\},\vspace{0.1cm}\\
\bar{u}=0, &\text{on}\; \partial_L \mathcal{D}.
\end{array}%
\right.
$$
Then, for $\gamma\in(1,\frac{2N+2}{2N+1})$, there exists a constant $C>0$ such that
$$\Big(\ds-\!\!\!\!\!\!\int_{B_r^{N+1}(x)}t^{1-2s}|\bar{u}|^\gamma dxdt\Big)^{\frac{1}{\gamma}}
\leq
C\ds-\!\!\!\!\!\!\int_{B_1^{N+1}(x)}t^{1-2s}|\bar{u}|^\gamma dxdt +C\ds\int_r^1\Big(\frac{1}{\rho^{N-2s}}\ds\int_{B_\rho^{N}(x)}f(y)dy\Big)\frac{d\rho}{\rho}
$$
holds for any $x\in\Om$ and $r\in(0,r_0)$, where $r_0=\text{dist}\,(x,\partial\Om)$.
\end{lem}

 Now we come to our main result in this section.

\begin{prop}\label{prop3.1}
Let $w_{n}$ be a weak solution of \eqref{h2.4}. Then there is a constant $C>0$
independent of $n$, such that
$$
\int_{\mathcal {A}^{N}_{n,2}}|w_{n}|^{p}dx\leq
C\Lambda_{n}^{\frac{Np}{2p_{1}}-\frac{N}{2}},
$$
and
$$
\int_{\mathcal {A}^{N+1}_{n,2}}t^{1-2s}|w_{n}|^{p}dxdt\leq
C\Lambda_{n}^{\frac{Np}{2p_{1}}-\frac{N+2-2s}{2}},
$$
where $p_{1}>2^{*}_s$ and $p\geq 2$ are any constants, satisfying
$$
p,p_{1}<\frac{2^*_s\sqrt{\bar{\mu}}}{\sqrt{\bar{\mu}}-\sqrt{\bar{\mu}-\mu}}.
$$
\end{prop}

In order to prove Proposition \ref{prop3.1}, we need the following
lemma.

\begin{lem}\label{lem3.2}
Suppose that $w_{n}$ satisfies \eqref{h2.4} with
$\epsilon=\epsilon_{n}\rightarrow 0$. Then there is a constant $C>0$
independent of $n,$ such that
$$
\sup\limits_{n}-\!\!\!\!\!\!\int_{B^{N+1}_{r}(y)}t^{1-2s}|\bar{w}_n|^\gamma dxdt\leq
C\Lambda_{n}^{\frac{N\gamma}{2p_{1}}} ,\,\,\,\forall y\in \R^{N},
$$
for all
$r\in\big[\bar{C}\Lambda_{n}^{-\frac{1}{2}},(\bar{C}+5)\Lambda_{n}^{-\frac{1}{2}}\big],$
where $\gamma\in(1,\frac{2N+2}{2N+1})$ and $p_{1}>2^{*}_s$ is any constant satisfying $p_{1}<\frac{2^*_s\sqrt{\bar{\mu}}}{\sqrt{\bar{\mu}}-\sqrt{\bar{\mu}-\mu}}.$
\end{lem}

\begin{proof}
Firstly, using H\"{o}lder inequality and \eqref{3.1.1.1}, we find
\begin{eqnarray*}
&&\ds\int_{B^{N+1}_1(y)}t^{1-2s}|\bar{w}_n|^\gamma dxdt\\
&\leq&\Big(\int_{B^{N+1}_1(y)}t^{1-2s}|\bar{w}_n|^\frac{2(N+1)}{N} dxdt\Big)^{\frac{N\gamma}{2(N+1)}}
\Big(\int_{B^{N+1}_1(y)}t^{1-2s}dxdt\Big)^{1-\frac{N\gamma}{2(N+1)}}\\
&\leq& C\Big(\int_{B^{N+1}_1(y)}t^{1-2s}|\nabla \bar{w}_n(x,t)|^2 dxdt\Big)^{\frac{1}{2}}\leq C.
\end{eqnarray*}
So it follows from \eqref{h2.4} and Lemma \ref{lem3.1.2} that
\begin{equation}\label{3.4.1}
\begin{split}
&\Big(-\!\!\!\!\!\!\int_{B^{N+1}_{r}(y)}t^{1-2s}|\bar{w}_n|^\gamma dxdt\Big)^{\frac{1}{\gamma}}\\
&\leq
C+C\ds\int_r^1\frac{1}{\rho^{N-2s}}\ds\int_{B_\rho^{N}(y)}\Big(\frac{\mu}{|x|^{2s}}w_{n}+2w_n^{2^*_s-1}+A\Big)dx\frac{d\rho}{\rho}\\
&\leq C+C\ds\int_r^1\frac{1}{\rho^{N-2s+1}}\ds\int_{B_\rho^{N}(y)}\frac{\mu w_{n}}{|x|^{2s}}dxd\rho
+C\ds\int_r^1\frac{1}{\rho^{N-2s+1}}\ds\int_{B_\rho^{N}(y)}w_n^{2^*_s-1}dxd\rho.
\end{split}
\end{equation}
By Proposition \ref{prop2.1}, we know that
$\|w_{n}\|_{*,p_{1},p_{2},\Lambda_{n}}\leq C$ for any $p_{1},p_{2}\in \big(\frac{2^{*}_s}{2},\frac{2^*_s\sqrt{\bar{\mu}}}{\sqrt{\bar{\mu}}-\sqrt{\bar{\mu}-\mu}}\big)$
satisfying $p_{2}<2^{*}_s<p_{1}$.

Let $p_1$ be a constant satisfying $2^*_s<p_1<\frac{2^*_s\sqrt{\bar{\mu}}}{\sqrt{\bar{\mu}}-\sqrt{\bar{\mu}-\mu}}$
and $p_2=\frac{2^*_s}{2}+\theta$, where $\theta>0$ is a small constant. Then we can choose $v_{1,n}$ and $v_{2,n}$
such that $w_n\leq v_{1,n}+v_{2,n}$,
$\|v_{1,n}\|_{*,p_1}\leq C$
and
$\|v_{2,n}\|_{*,p_2}\leq C\Lambda_{n}^{\frac{N}{2^{*}_s}-\frac{N}{p_{2}}}.$ So,
\begin{equation*}\label{3.2.1}
\begin{array}{ll}
&\ds\int_r^1\frac{1}{\rho^{N-2s+1}}\ds\int_{B_\rho^{N}(y)}\frac{\mu v_{1,n}}{|x|^{2s}}dxd\rho\vspace{0.2cm}\\
&\leq \ds\int_r^1\frac{1}{\rho^{N-2s+1}}\Big(\ds\int_{B_\rho^{N}(y)}\mu\frac{|v_{1,n}|^{\frac{2p_1}{2^*_s}}}{|x|^{2s}}dx\Big)^{\frac{2^*_s}{2p_1}}
\Big(\ds\int_{B_\rho^{N}(y)}\frac{\mu}{|x|^{2s}}dx\Big)^{1-\frac{2^*_s}{2p_1}}\vspace{0.2cm}\\
&\leq C\ds\int_r^1\rho^{2s-N-1+(N-2s)(1-\frac{2^{*}_s}{2p_{1}})}d\rho
\leq  C r^{-\frac{(N-2s)2^{*}_s}{2p_{1}}}\leq C\Lambda_{n}^{\frac{(N-2s)2^{*}_s}{4p_{1}}}=C\Lambda_{n}^{\frac{N}{2p_1}},
\end{array}
\end{equation*}
and
\begin{equation*}\label{3.2.2}
\begin{array}{ll}
&\ds\int_r^1\frac{1}{\rho^{N-2s+1}}\ds\int_{B_\rho^{N}(y)}\frac{\mu v_{2,n}}{|x|^{2s}}dxd\rho\vspace{0.2cm}\\
&\leq \ds\int_r^1\frac{1}{\rho^{N-2s+1}}\Big(\ds\int_{B_\rho^{N}(y)}\frac{\mu|v_{2,n}|^{\frac{2p_2}{2^*_s}}}{|x|^{2s}}dx\Big)^{\frac{2^*_s}{2p_2}}
\Big(\ds\int_{B_\rho^{N}(y)}\frac{\mu}{|x|^{2s}}dx\Big)^{1-\frac{2^*_s}{2p_2}}\vspace{0.2cm}\\
&\leq C\Lambda_{n}^{\frac{N}{2^*_s}-\frac{N}{p_2}}\ds\int_r^1\rho^{2s-N-1+(N-2s)(1-\frac{2^{*}_s}{2p_{2}})}d\rho
\leq  C \Lambda_{n}^{\frac{N}{2^*_s}-\frac{N}{2p_2}}=C\Lambda_{n}^{\theta_1},
\end{array}
\end{equation*}
where $\theta_1>0$ is a small constant if we choose $\theta>0$ small enough.

Thus, we obtain that
\begin{equation}\label{3.2.3}
\begin{array}{ll}
&\ds\int_r^1\frac{1}{\rho^{N-2s+1}}\ds\int_{B_\rho^{N}(y)}\frac{\mu w_{n}}{|x|^{2s}}dxd\rho\vspace{0.2cm}\\
&\leq\ds\int_r^1\frac{1}{\rho^{N-2s+1}}\ds\int_{B_\rho^{N}(y)}\frac{\mu v_{1,n}}{|x|^{2s}}dxd\rho+
\ds\int_r^1\frac{1}{\rho^{N-2s+1}}\ds\int_{B_\rho^{N}(y)}\frac{\mu v_{2,n}}{|x|^{2s}}dxd\rho\vspace{0.2cm}\\
&\leq C\Lambda_{n}^{\frac{N}{2p_{1}}}+C\Lambda_{n}^{\theta_{1}}
\leq C\Lambda_{n}^{\frac{N}{2p_{1}}}.
\end{array}
\end{equation}

Let $\bar{p}_{2}=2^{*}_s-1=\frac{N+2s}{N-2s}$ and let $p_{1}>2^{*}_s$
with $p_{1}<\frac{2^*_s\sqrt{\bar{\mu}}}{\sqrt{\bar{\mu}}-\sqrt{\bar{\mu}-\mu}}$. Then we can
choose $\bar{v}_{1,n}$ and $\bar{v}_{2,n}$ such that $|w_{n}|\leq \bar{v}_{1,n}+\bar{v}_{2,n}$
and $||\bar{v}_{1,n}||_{*,p_{1}}\leq C$ and $||\bar{v}_{2,n}||_{*,\bar{p}_{2}}\leq C\Lambda_{n}^{\frac{N}{2^{*}_s}-\frac{N}{\bar{p}_{2}}}.$
Since $p_{1}>2^{*}_s$, we know that $\frac{N(N+2s)}{2p_{1}(N-2s)}-s<\frac{N}{2p_{1}}$. Therefore,
\begin{eqnarray*}
\ds\int_r^1\frac{1}{\rho^{N-2s+1}}\ds\int_{B_\rho^{N}(y)}\bar{v}_{1,n}^{2^*_s-1}dxd\rho
&\leq&\ds\int_r^1\frac{1}{\rho^{N-2s+1}}\Big(\ds\int_{B_\rho^{N}(y)}|\bar{v}_{1,n}|^{p_1}dx\Big)^{\frac{N+2s}{(N-2s)p_1}}
\rho^{N(1-\frac{N+2s}{p_1(N-2s)})}d\rho\\
&\leq&C\int_{r}^{1}\rho^{2s-1-\frac{N(N+2s)}{p_1(N-2s)})}d\rho
\leq \Lambda_{n}^{\frac{N(N+2s)}{2p_1(N-2s)}-s}
\leq C\Lambda_{n}^{\frac{N}{2p_{1}}},
\end{eqnarray*}
and
\begin{eqnarray*}
\ds\int_r^1\frac{1}{\rho^{N-2s+1}}\ds\int_{B_\rho^{N}(y)}\bar{v}_{2,n}^{2^*_s-1}dxd\rho
&\leq&\int_r^1\frac{1}{\rho^{N-2s+1}}
(\Lambda_{n}^{\frac{N}{2^{*}_s}-\frac{N}{\bar{p_2}}})^{\frac{N+2s}{N-2s}}d\rho\\
&\leq&C\Lambda_{n}^{\frac{2s-N}{2}}
\int_r^1\frac{1}{\rho^{N-2s+1}}d\rho
 \leq C.
\end{eqnarray*}

Hence, we get
\begin{equation}\label{together}
\begin{array}{ll}
&\ds\int_r^1\frac{1}{\rho^{N-2s+1}}\ds\int_{B_\rho^{N}(y)}w_n^{2^*_s-1}dxd\rho\vspace{0.2cm}\\
&\leq
C\ds\int_r^1\frac{1}{\rho^{N-2s+1}}\ds\int_{B_\rho^{N}(y)}\bar{v}_{1,n}^{2^*_s-1}dxd\rho+
C\ds\int_r^1\frac{1}{\rho^{N-2s+1}}\ds\int_{B_\rho^{N}(y)}\bar{v}_{2,n}^{2^*_s-1}dxd\rho\vspace{0.2cm}\\
&\leq C\Lambda_{n}^{\frac{N}{2p_{1}}}.
\end{array}
\end{equation}

 From \eqref{together}, \eqref{3.4.1} and \eqref{3.2.3}, we have
$$
\Big(-\!\!\!\!\!\!\int_{B^{N+1}_{r}(y)}t^{1-2s}|\bar{w}_n|^\gamma dxdt\Big)^{\frac{1}{\gamma}}
\leq
C\Lambda_{n}^{\frac{N}{2p_{1}}}
$$
for any
$r\in\big[\bar{C}\Lambda_{n}^{-\frac{1}{2}},(\bar{C}+5)\Lambda_{n}^{-\frac{1}{2}}\big].$
\end{proof}

\begin{proof}[\textbf{Proof of Proposition \ref{prop3.1}}]

It follows from Lemma \ref{lem3.2} that for any $y\in \mathcal
{A}^{N}_{n,2}$ and $r\in\big[\bar{C}\Lambda_{n}^{-\frac{1}{2}},(\bar{C}+5)\Lambda_{n}^{-\frac{1}{2}}\big]$, we get that
if $\gamma\in(1,\frac{2N+2}{2N+1})$,
\begin{equation}\label{3.3}
\begin{array}{ll}
\ds\int_{B^{N+1}_{r}(y)}t^{1-2s}|\bar{w}_n|^\gamma dxdt
\leq C\Lambda_{n}^{\frac{N\gamma}{2p_{1}}}\int_{B^{N+1}_{r}(y)}t^{1-2s}dxdt
\leq C\Lambda_{n}^{\frac{N\gamma}{2p_{1}}}\Lambda_{n}^{-\frac{1}{2}(N+2-2s)}.
\end{array}
\end{equation}

%On the other hand,
%\begin{equation}\label{3.4}
%\begin{array}{ll}
%\ds\int_{B^{N+1}_{r}(y)}t^{1-2s}|\bar{w}_n|^\gamma
%\geq C\ds\int_{B^{N+1}_{r}(y)}t^{1-2s}|\bar{w}_n(x,0)|^\gamma dxdt
%\geq C\Lambda_{n}^{-\frac{1}{2}(2-2s)}\ds\int_{B^{N}_{r}(y)}|\bar{w}_n(x)|^\gamma dx.
%\end{array}
%\end{equation}
%Then it follows from \eqref{3.3} and \eqref{3.4} that
%\begin{equation}\label{3.5}
%\begin{array}{ll}
%\ds\int_{B^{N}_{r}(y)}|w_n(x)|^\gamma dx
%\leq C\Lambda_{n}^{\frac{1}{2}(2-2s)}\ds\int_{B^{N+1}_{r}(y)}t^{1-2s}|\bar{w}_n|^\gamma dxdt
%\leq C\Lambda_{n}^{\frac{N\gamma}{2p_{1}}-\frac{N}{2}}.
%\end{array}
%\end{equation}

Let $\bar{v}_{n}(z)=\bar{w}_{n}\big(\Lambda_{n}^{-\frac{1}{2}}z\big),z\in\mathcal{D}_{n}$,
where $\mathcal{D}_{n}=\big\{z:\Lambda_{n}^{-\frac{1}{2}}z\in\mathcal{D}\big\}.$

Then $\bar{v}_{n}$ satisfies
$$
\left\{%
\begin{array}{ll}
    div(t^{1-2s}\bar{v}_n)=0,  & \hbox{$\text{in}~ \mathcal{D}_n$},\vspace{0.1cm} \\
   \mathcal{A}_s \bar{v}_{n}(x,0)\leq\ds\frac{\mu v_{n}}{|x|^{2s}}+ \Lambda_{n}^{-s}\big(|v_{n}|^{2^{*}_s-2}+a\big)v_{n},
    & \hbox{$\text{on}~ \Omega_{n}$},\\
\end{array}%
\right.
$$
where $\Omega_{n}=\big\{x:\Lambda_{n}^{-\frac{1}{2}}x\in\Om\big\}$.

Let $\bar{x}=\Lambda_{n}^{\frac{1}{2}}y.$ Since
$B_{\Lambda_{n}^{-\frac{1}{2}}}(y),y\in \mathcal
{A}^{N}_{n,2}$ does not contain any concentration point of $u_{n}$, we
can deduce that
\begin{eqnarray*}
\int_{B_{1}^N(\bar{x})}|\Lambda_{n}^{-s}(|v_{n}|^{2^{*}_s-2}+a)|^{\frac{N}{2s}}dx
\leq
C\int_{B^N_{\Lambda_{n}^{-\frac{1}{2}}}(y)}|u_{n}|^{2^{*}_s}dx+C\Lambda_{n}^{-\frac{N}{2}}
\rightarrow& 0,
\end{eqnarray*}
as $n\rightarrow\infty$.
Thus by Lemma \ref{lema.4} and \eqref{3.3}, we obtain
\begin{eqnarray*}
||v_{n}||_{L^p(B^N_{\frac{1}{2}}(\bar{x}))} \leq
C\Big(\int_{B^{N+1}_{1}(\bar{x})}t^{1-2s}|\bar{v}_{n}|^{\ga}dxdt\Big)^{\frac{1}{\ga}}
\leq C\Big(\Lambda_{n}^{\frac{N+2-2s}{2}}
\int_{B^{N+1}_{\Lambda_{n}^{-\frac{1}{2}}}(y)}t^{1-2s}|\bar{w}_{n}|^{\ga}dxdt\Big)^{\frac{1}{\ga}}
\leq C\Lambda_{n}^{\frac{N}{2p_{1}}},
\end{eqnarray*}
and
$$ \Big(\int_{B^{N+1}_{\frac{1}{2}}(\bar{x})}t^{1-2s}|\bar{v}_{n}|^{p}dxdt\Big)^{\frac{1}{p}}\leq
C\Big(\int_{B^{N+1}_{1}(\bar{x})}t^{1-2s}|\bar{v}_{n}|^{\ga}dxdt\Big)^{\frac{1}{\ga}}
\leq C\Lambda_{n}^{\frac{N}{2p_{1}}},
$$
for any $p>\max\{2^{*}_s,2^{\sharp}\}$ with $p<\min\{\frac{2^*_s\sqrt{\bar{\mu}}}{\sqrt{\bar{\mu}}-\sqrt{\bar{\mu}-\mu}},
\frac{2^{\sharp}\sqrt{\bar{\mu}}}{\sqrt{\bar{\mu}}-\sqrt{\bar{\mu}-\mu}}\}$ and $2^{\sharp}=\frac{2(N+1)}{N}$.

As a result,
$$
\Lambda_{n}^{\frac{N}{2}}\int_{B^N_{\frac{1}{2}\Lambda_{n}^{-\frac{1}{2}}}(y)}|w_{n}|^{p}dx\leq
C\Lambda_{n}^{\frac{pN}{2p_{1}}},\,\,\forall y\in\mathcal {A}^{N}_{n,2},
$$
and
 $$
 \Lambda_{n}^{\frac{N+2-2s}{2}}\int_{B^{N+1}_{\frac{1}{2}\Lambda_{n}^{-\frac{1}{2}}}(y)}t^{1-2s}|\bar{w}_{n}|^{p}dxdt\leq
C\Lambda_{n}^{\frac{pN}{2p_{1}}},\,\,\forall y\in\mathcal {A}^{N}_{n,2}.
 $$

 Hence, for any $p>\max\{2^{*}_s,2^{\sharp}\}$ with $p<\min\{\frac{2^*_s\sqrt{\bar{\mu}}}{\sqrt{\bar{\mu}}-\sqrt{\bar{\mu}-\mu}},
\frac{2^{\sharp}\sqrt{\bar{\mu}}}{\sqrt{\bar{\mu}}-\sqrt{\bar{\mu}-\mu}}\}$,
 there holds
$$
\int_{\mathcal {A}^{N}_{n,2}}|w_{n}|^{p} dx\leq
C\Lambda_{n}^{\frac{Np}{2p_{1}}-\frac{N}{2}},
$$
and
$$
\int_{\mathcal {A}^{N+1}_{n,2}}t^{1-2s}|\bar{w}_{n}|^{p}dxdt
C\Lambda_{n}^{\frac{Np}{2p_{1}}-\frac{N+2-2s}{2}}.
$$

On the other hand, for any $2\leq p\leq \max\{2^{*}_s,2^{\sharp}\},$ take $\bar{p}>\max\{2^{*}_s,2^{\sharp}\}$ and $\bar{p}<\min\{\frac{2^*_s\sqrt{\bar{\mu}}}{\sqrt{\bar{\mu}}-\sqrt{\bar{\mu}-\mu}},
\frac{2^{\sharp}\sqrt{\bar{\mu}}}{\sqrt{\bar{\mu}}-\sqrt{\bar{\mu}-\mu}}\}$. Then
\begin{eqnarray*}
\int_{\mathcal {A}^{N}_{n,2}}|w_{n}|^{p}dx
\leq C\Bigl(\int_{\mathcal {A}^{N}_{n,2}}|w_{n}|^{\bar{p}}dx\Bigl)^{\frac{p}{\bar{p}}}\Lambda_{n}^{-\frac{N}{2}(1-\frac{p}{\bar{p}})}
\leq C\Lambda_{n}^{\frac{Np}{2p_{1}}-\frac{N}{2}},
\end{eqnarray*}
and
\begin{eqnarray*}
\int_{\mathcal {A}^{N+1}_{n,2}}t^{1-2s}|\bar{w}_{n}|^{p}dxdt
\leq C\Bigl(\int_{\mathcal {A}^{N+1}_{n,2}}t^{1-2s}|\bar{w}_{n}|^{\bar{p}}dxdt\Bigl)^{\frac{p}{\bar{p}}}\Lambda_{n}^{-\frac{N+2-2s}{2}(1-\frac{p}{\bar{p}})}
\leq C\Lambda_{n}^{\frac{Np}{2p_{1}}-\frac{N+2-2s}{2}},
\end{eqnarray*}
Hence, for any $p\geq2$,
$$
\Big(\int_{\mathcal {A}^{N}_{n,2}}|w_{n}|^{p}dx\Big)^{\frac{1}{p}}\leq C\Lambda_{n}^{\frac{N}{2p_{1}}-\frac{N}{2p}},
$$
and
\begin{eqnarray*}
\int_{\mathcal {A}^{N+1}_{n,2}}t^{1-2s}|\bar{w}_{n}|^{p}dxdt
\leq C\Lambda_{n}^{\frac{Np}{2p_{1}}-\frac{1}{2}(N+2-2s)}.
\end{eqnarray*}
\end{proof}

Let $\mathcal
{A}^{d}_{n,3}=\Big(B^d_{(\bar{C}+3)\Lambda_{n}^{-\frac{1}{2}}}(x_{n})\backslash
B^d_{(\bar{C}+2)\Lambda_{n}^{-\frac{1}{2}}}(x_{n})\Big)\cap\Omega\,(\text{or}\,\mathcal{D}),\,\,d=N, N+1.$

\begin{prop}\label{prop3.3}
We have
\begin{equation}\label{3.3.1}
\int_{\mathcal{A}^{N+1}_{n,3}}t^{1-2s}|\nabla\bar{ u}_{n}(x,t)|^{2}dxdt,
\int_{\mathcal{A}^{N}_{n,3}}\frac{\mu| u_{n}|^{2}}{|x|^{2s}}dx
\leq C\int_{\mathcal{A}^{N}_{n,2}}\big(|u_n|^{2^*_s}+1\big)dx
+C\Lambda_{n}\int_{\mathcal{A}^{N+1}_{n,2}}t^{1-2s}|\bar{ u}_{n}(x,t)|^{2}dxdt.
\end{equation}

In particular,
\begin{equation}\label{3.3.2}
\int_{\mathcal
{A}^{N+1}_{n,3}}t^{1-2s}|\nabla\bar{ u}_{n}(x,t)|^{2}dxdt,
\int_{\mathcal
{A}^{N}_{n,3}}\frac{\mu| u_{n}|^{2}}{|x|^{2s}}dx\leq
C\Lambda_{n}^{\frac{2s-N}{2}+\frac{N}{p_{1}}}.
\end{equation}
\end{prop}

\begin{proof}
Let $\bar{\phi}_{n}\in C^{\infty}_{0}(\mathcal {A}^{N+1}_{n,2})$ be a function
with $\bar{\phi}_{n}=1$ in $\mathcal {A}^{N+1}_{n,3}$, $0\leq\bar{\phi}_{n}\leq 1$
and $|\nabla \bar{\phi}_{n}|\leq C\Lambda_{n}^{\frac{1}{2}}.$

From
$$
\int_{\mathcal{D}}t^{1-2s}\nabla  \bar{u}_n\nabla (\bar{\phi}_{n}^{2}\bar{u}_{n})dxdt
-\int_{\Omega}\frac{\mu \phi_{n}^{2}u^{2}_{n}}{|x|^{2s}} dx\leq
\int_{\Omega}(2|u_{n}|^{2^{*}_s-1}+A)\phi_{n}^{2}|u_{n}|dx,
$$
we can prove \eqref{3.3.1}.
Since $p_{1}>2^{*}_s$, we see
$$
\frac{2^*_sN}{2p_{1}}-\frac{N}{2}
<\frac{N}{p_{1}}-\frac{N-2s}{2}.
$$ Thus
from \eqref{3.3.1} and Proposition \ref{prop3.1}, we have
$$
\int_{\mathcal
{A}^{N+1}_{n,3}}t^{1-2s}|\nabla\bar{ u}_{n}(x,t)|^{2}dxdt,
\int_{\mathcal
{A}^{N}_{n,3}}\frac{\mu| u_{n}|^{2}}{|x|^{2s}}dx \leq
C\Lambda_{n}^{\frac{2^*_sN}{2p_{1}}-\frac{N}{2}}+C\Lambda_{n}^{\frac{N}{p_{1}}-\frac{N-2s}{2}}\leq
\tilde{C}\Lambda_{n}^{\frac{2s-N}{2}+\frac{N}{p_{1}}}.
$$
\end{proof}

\section{Proof of the Main Result}

Choose an $\ell_{n}\in[\bar{C}+2,\bar{C}+3]$ such that
\begin{equation}\label{4.1}
\begin{array}{ll}
&\ds\int_{\partial
B^N_{\ell_{n}\Lambda_{n}^{-\frac{1}{2}}}(x_{n})}\Bigl(\Lambda_{n}^{-1}|u_{n}|^{2^{*}_s-\epsilon_{n}}
+|u_{n}|^{2}+\Lambda_{n}^{-1}\frac{\mu u_{n}^{2}}{|x|^{2s}} \Bigl)dS_x\vspace{0.2cm}\\
&\leq C\Lambda_{n}^{\frac{2s-1}{2}}\ds \int_{\mathcal
{A}^{N}_{n,3}}\Bigl(\Lambda_{n}^{-1}|u_{n}|^{2^{*}_s-\epsilon_{n}}
+|u_{n}|^{2}+\Lambda_{n}^{-1}\frac{\mu u_{n}^{2}}{|x|^{2s}} \Bigl)dx,
\end{array}
\end{equation}
and
\begin{equation}\label{4.1.1}
\ds\int_{\partial
B^{N+1}_{\ell_{n}\Lambda_{n}^{-\frac{1}{2}}}(x_{n})}\Lambda_{n}^{-1}t^{1-2s}\big(|\nabla\bar{ u}_{n}|^{2}+|\bar{ u}_{n}|^2\big)dS_z
\leq C\Lambda_{n}^{\frac{2s-1}{2}}\ds \int_{\mathcal
{A}^{N+1}_{n,3}}\Lambda_{n}^{-1}t^{1-2s}\big(|\nabla\bar{ u}_{n}|^{2}+|\bar{ u}_{n}|^2\big)dxdt.
\end{equation}
Applying Proposition \ref{prop3.1}, \eqref{4.1}, \eqref{4.1.1} and \eqref{3.3.2},
we get
\begin{equation}\label{4.2}
\begin{array}{ll}
&\ds\int_{\partial
B^N_{\ell_{n}\Lambda_{n}^{-\frac{1}{2}}}(x_{n})}\Bigl(\Lambda_{n}^{-1}|u_{n}|^{2^{*}_s-\epsilon_{n}}
+|u_{n}|^{2}+\Lambda_{n}^{-1}\frac{\mu u_{n}^{2}}{|x|^{2s}} \Bigl)dS_x
\vspace{0.2cm}\\&+
\ds\int_{\partial
B^{N+1}_{\ell_{n}\Lambda_{n}^{-\frac{1}{2}}}(x_{n})}\Lambda_{n}^{-1}t^{1-2s}\big(|\nabla\bar{ u}_{n}|^{2}+|\bar{ u}_{n}|^2\big)dS_z\vspace{0.2cm}\\
&\leq
C\Lambda_{n}^{\frac{2s-1}{2}}\Bigl(C\Lambda_{n}^{-1}\Lambda_{n}^{\frac{N(2^{*}_s-\epsilon_{n})}{2p_{1}}
-\frac{N}{2}}+C\Lambda_{n}^{\frac{N}{p_{1}}-\frac{N}{2}}
+C\Lambda_{n}^{-1}
\Lambda_{n}^{\frac{N}{p_{1}}+\frac{2s-N}{2}}\Bigl)\vspace{0.2cm}\\
&\leq C\Lambda_{n}^{\frac{2s-1-N}{2}+\frac{N}{p_{1}}},
\end{array}
\end{equation}
since $-1+\frac{N(2^{*}_s-\epsilon_{n})}{2p_{1}}
-\frac{N}{2}<-\frac{N}{2}+\frac{N}{p_{1}}.$

\begin{proof}[\textbf{Proof of Theorem \ref{thm1.2}}]

We have three different cases: (i)
$B^N_{\ell_{n}\Lambda_{n}^{-\frac{1}{2}}}(x_{n})\cap(\R^{N}\backslash\Omega)\neq\emptyset$;
(ii) $B^N_{\ell_{n}\Lambda_{n}^{-\frac{1}{2}}}(x_{n})\subset\Omega$ and $0\not\in \overline{B_{\ell_{n}\Lambda_{n}^{-\frac{1}{2}}}(x_{n})}$; (iii)$B^N_{\ell_{n}\Lambda_{n}^{-\frac{1}{2}}}(x_{n})\subset\Omega$ and $0\in \overline{B_{\ell_{n}\Lambda_{n}^{-\frac{1}{2}}}(x_{n})}$.

Firstly, define
$\partial_+\mathcal{F}=\big\{z=(x,t)\in\R^{N+1}_{+}:\,(x,t)\in\partial\mathcal{F}, \, t>0\big\}$ and
$\partial_b\mathcal{F}=\big\{x\in\R^N:\,(x,0)\in\partial\mathcal{F}\cap \R^N\times\{0\}\big\}$.
 Let $B_n^N=B^N_{\ell_{n}\Lambda_{n}^{-\frac{1}{2}}}(x_{n})\cap \Om$, $B^{N+1}_n=B^{N+1}_{\ell_{n}\Lambda_{n}^{-\frac{1}{2}}}(x_{n})\cap \mathcal{D}$,
 and $p_{n}=2^{*}_s-\epsilon_{n}$. Then from Proposition \ref{propc.1}, we have the following local
Pohozaev identity for $u_{n}$ on $B^{N+1}_n$,
\begin{equation}\label{4.3}
\begin{array}{ll}
&\Big(\ds\frac{N}{p_n}-\ds\frac{N-2s}{2}\Big)\ds\int_{B_n^N}|u_n|^{p_n}dx+sa\ds\int_{B_n^N}|u_n|^2dx
+s\mu\ds\int_{B_n^N}\frac{x\cdot x_0|u_n|^2}{|x|^{2s+2}}dx\vspace{0.2cm}\\
&=\ds\frac{1}{2}\ds\int_{\partial B_n^N}\Big(a+\frac{\mu}{|x|^{2s}}\Big)|u_n|^2(x-x_0)\cdot\nu_xdS_x
+\frac{1}{p_n}\ds\int_{\partial B_n^N}|u_n|^{p_n}(x-x_0)\cdot\nu_xdS_x\vspace{0.2cm}\\
&\quad+\ds\int_{\partial_+B^{N+1}_n}t^{1-2s}\Big((z-z_0,\nabla\bar{u}_n)\nabla\bar{u}_n-(z-z_0)\frac{|\nabla\bar{u}_n|^2}{2},\nu_z\Big)dS_z\vspace{0.2cm}\\
&\quad+\ds\frac{N-2s}{2}\ds\int_{\partial_+B^{N+1}_n}t^{1-2s}\bar{u}_n\frac{\partial \bar{u}_n}{\partial\nu_z}dS_z,
\end{array}
\end{equation}
where $z_0=(x_0,0)$, $z=(x,t)$ and $x_0$ in \eqref{4.3} is chosen as follows. In case (i), we take
$x_{0}\in \R^{N}\backslash\Omega$ with $|x_{0}-x_{n}|\leq
2\ell_{n}\Lambda_{n}^{-\frac{1}{2}}$ and $\nu_x\cdot(x-x_{0})\leq 0$ in
$\partial \Omega\cap B_{n}^N$. Then we see from the fact $\nu_z=(\nu_x,0)$ that
$\nu_z\cdot(z-z_0)=\nu_x\cdot(x-x_0)\leq0$  and with this $x_{0}$, we can check that $x_{0}\cdot x\geq0$ in $ B_{n}^N$.
In case (ii), we take a point
$x_{0}=x_{n}$ Then  $x_{0}\cdot x\geq0$ in $ B_{n}^N$. In case (iii), we take $x_{0}=0$. Thus, in
any case $x_{0}\cdot x\geq0$ in $ B_{n}^N$.

In fact, in case (i) and case (ii), $u_n\in C^{2s}(B_n^N)$. So, \eqref{4.3} is the usual local Pohozaev identity.
Now we prove that \eqref{4.3} holds as well in case (iii).

To see this, since $\ds\int_\Om|(-\Delta)^{\frac{s}{2}}u_n|^2dx=\ds\int_{\mathcal{D}}t^{1-2s}|\nabla\bar{u}_n|^2dxdt\leq C$, we can
choose $\theta_j\rightarrow 0$ as $j\rightarrow +\infty$ such that
\begin{equation}\label{4.3.1}
\theta_j\ds\int_{\pa_+ B^{N+1}_{\theta_j}(0)}t^{1-2s}|\nabla\bar{u}_n|^2dS_z+
\theta_j\ds\int_{\pa B^{N}_{\theta_j}(0)}|u_n|^{p_n}dS_x+\theta_j\ds\int_{\partial  B^{N}_{\theta_j}(0)}
\Big(a+\frac{\mu}{|x|^{2s}}\Big)|u_n|^2dS_x\rightarrow0.
\end{equation}
Let $B^{N+1}_{n,\theta_j}=B^{N+1}_n\backslash B^{N+1}_{\theta_j}(0)$. Then $u_n\in C^{2s}(\overline{B^{N}_{n,\theta_j}})$ and
\begin{equation}\label{4.3.2}
\begin{array}{ll}
&\Big(\ds\frac{N}{p_n}-\ds\frac{N-2s}{2}\Big)\ds\int_{B^{N}_{n,\theta_j}}|u_n|^{p_n}dx+sa\ds\int_{B^{N}_{n,\theta_j}}|u_n|^2dx
+s\mu\ds\int_{B^{N}_{n,\theta_j}}\frac{x\cdot x_0|u_n|^2}{|x|^{2s+2}}dx\vspace{0.2cm}\\
&=\ds\frac{1}{2}\ds\int_{\partial B^{N}_{n,\theta_j}}\Big(a+\frac{\mu}{|x|^{2s}}\Big)|u_n|^2x\cdot\nu_xdS_x
+\frac{1}{p_n}\ds\int_{\partial B^{N}_{n,\theta_j}}|u_n|^{p_n}x\cdot\nu_xdS_x\vspace{0.2cm}\\
&\quad+\ds\int_{\partial_+B^{N+1}_{n,\theta_j}}t^{1-2s}\Big((z,\nabla\bar{u}_n)
\nabla\bar{u}_n-z\frac{|\nabla\bar{u}_n|^2}{2},\nu_z\Big)dS_z\vspace{0.2cm}\\
&\quad+\ds\frac{N-2s}{2}\ds\int_{\partial_+B^{N+1}_{n,\theta_j}}t^{1-2s}\bar{u}_n\frac{\partial \bar{u}_n}{\partial\nu_z}dS_z.
\end{array}
\end{equation}
From \eqref{4.3.1} and Proposition \ref{propbb3}, we have
\begin{equation}\label{4.3.3}
\begin{array}{ll}
\Big|\ds\int_{\partial_+B^{N+1}_{\theta_j}(0)}t^{1-2s}\bar{u}_n\frac{\partial \bar{u}_n}{\partial\nu_z}dS_z\Big|
&\leq \Big(\ds\int_{\partial_+B^{N+1}_{\theta_j}(0)}t^{1-2s}|\nabla\bar{u}_n|^2dS_z\Big)^{\frac{1}{2}}
\Big(\ds\int_{\partial_+B^{N+1}_{\theta_j}(0)}t^{1-2s}|\bar{u}_n|^2dS_z\Big)^{\frac{1}{2}}\vspace{0.2cm}\\
&\leq C\theta_j^{-\frac{1}{2}}\theta_j^{1+\beta}=o(1),
\end{array}
\end{equation}
and
\begin{equation}\label{4.3.4}
\begin{array}{ll}
&\ds\frac{1}{2}\ds\int_{\partial B^{N}_{\theta_j}(0)}\Big(a+\frac{\mu}{|x|^{2s}}\Big)|u_n|^2x\cdot\nu_xdS_x
+\frac{1}{p_n}\ds\int_{\partial B^{N}_{\theta_j}(0)}|u_n|^{p_n}x\cdot\nu_xdS_x\vspace{0.2cm}\\
&\quad+\ds\int_{\partial_+B^{N+1}_{\theta_j}(0)}t^{1-2s}\Big((z,\nabla\bar{u}_n)
\nabla\bar{u}_n-z\frac{|\nabla\bar{u}_n|^2}{2},\nu_z\Big)dS_z\vspace{0.2cm}\\
&=O\Big(\theta_j\ds\int_{\partial B^{N}_{\theta_j}(0)}|u_n|^{p_n}dS_x
+\theta_j\ds\int_{\partial B^{N}_{\theta_j}(0)}\big(a+\frac{\mu}{|x|^{2s}}\big)|u_n|^2dS_x
+\theta_j\ds\int_{\partial_+B^{N+1}_{\theta_j}(0)}t^{1-2s}|\nabla\bar{u}_n|^2dS_z\Big)=o(1).
\end{array}
\end{equation}
So, letting $j\rightarrow +\infty$ in \eqref{4.3.2}, from \eqref{4.3.3} and \eqref{4.3.4}, we can get \eqref{4.3}.

Since $p_{n}< 2^{*}_s,$
 the first term in the left hand side of \eqref{4.3} is
nonnegative and by the choice of $x_{0}$, the third term in the left hand side of \eqref{4.3}
is also nonnegative. Hence \eqref{4.3} can be rewritten as
\begin{equation}\label{4.4}
\begin{array}{ll}
sa\ds\int_{B_n^N}|u_n|^2dx
&\leq\ds\frac{1}{2}\ds\int_{\partial B_n^N}\Big(a+\frac{\mu}{|x|^{2s}}\Big)|u_n|^2(x-x_0)\cdot\nu_xdS_x
+\frac{1}{p_n}\ds\int_{\partial B_n^N}|u_n|^{p_n}(x-x_0)\cdot\nu_xdS_x\vspace{0.2cm}\\
&\quad+\ds\int_{\partial_+B^{N+1}_n}t^{1-2s}\Big((z-z_0,\nabla\bar{u}_n)\nabla\bar{u}_n-(z-z_0)\frac{|\nabla\bar{u}_n|^2}{2},\nu_z\Big)dS_z\vspace{0.2cm}\\
&\quad+\ds\frac{N-2s}{2}\ds\int_{\partial_+B^{N+1}_n}t^{1-2s}\bar{u}_n\frac{\partial \bar{u}_n}{\partial\nu_z}dS_z.
\end{array}
\end{equation}

Now we decompose $\partial B^N_{n}$ into $\partial B^N_{n}=\partial_{i}
B^N_{n}\cup\partial_{e} B^N_{n}$, where $\partial_{i} B^N_{n}=\partial
B^N_{n}\cap\Omega$ and $\partial_{e} B^N_{n}=\partial
B^N_{n}\cap\partial \Omega.$ Similarly,
$\pa_+B^{N+1}_n=\partial_{i}
B^{N+1}_{n}\cup\partial_{e} B^{N+1}_{n}$, where $\partial_{i} B^{N+1}_{n}=\pa_+B^{N+1}_n
\cap\mathcal{D}$ and $\partial_{e} B^{N+1}_{n}=\pa_+B^{N+1}_n
\cap\pa\mathcal{D}.$

Observing that $u_{n}=0$ on $\partial_{e} B^N_{n}$ and $\bar{u}_n=0$ on $\partial_{e} B^{N+1}_{n}$, we have
\begin{eqnarray*}
&&\ds\frac{1}{2}\ds\int_{\partial_e B_n^N}\Big(a+\frac{\mu}{|x|^{2s}}\Big)|u_n|^2(x-x_0)\cdot\nu_xdS_x
+\frac{1}{p_n}\ds\int_{\partial_e B_n^N}|u_n|^{p_n}(x-x_0)\cdot\nu_xdS_x\\
&&\quad+\ds\frac{N-2s}{2}\ds\int_{\partial_eB^{N+1}_n}t^{1-2s}\bar{u}_n\frac{\partial \bar{u}_n}{\partial\nu_z}dS_z=0.
\end{eqnarray*}

Also, noting that $\nabla \bar{u}_n=\pm |\nabla \bar{u}_n|\nu_z$ on $\partial_{e} B^{N+1}_{n}$, we find
\begin{eqnarray*}
&&\ds\int_{\partial_eB^{N+1}_n}t^{1-2s}\Big((z-z_0,\nabla\bar{u}_n)\nabla\bar{u}_n-(z-z_0)\frac{|\nabla\bar{u}_n|^2}{2},\nu_z\Big)dS_z\\
&&=\ds\int_{\partial_eB^{N+1}_n}t^{1-2s}\frac{|\nabla\bar{u}_n|^2}{2}(z-z_0,\nu_z)dS_z\leq0.
\end{eqnarray*}

Hence, we can rewrite \eqref{4.4} as
\begin{equation}\label{4.9}
\begin{array}{ll}
sa\ds\int_{B_n^N}|u_n|^2dx
&\leq\ds\frac{1}{2}\ds\int_{\partial_i B_n^N}\Big(a+\frac{\mu}{|x|^{2s}}\Big)|u_n|^2(x-x_0)\cdot\nu_xdS_x
+\frac{1}{p_n}\ds\int_{\partial_i B_n^N}|u_n|^{p_n}(x-x_0)\cdot\nu_xdS_x\vspace{0.2cm}\\
&\quad+\ds\int_{\partial_iB^{N+1}_n}t^{1-2s}\Big((z-z_0,\nabla\bar{u}_n)\nabla\bar{u}_n-(z-z_0)\frac{|\nabla\bar{u}_n|^2}{2},\nu_z\Big)dS_z\vspace{0.2cm}\\
&\quad+\ds\frac{N-2s}{2}\ds\int_{\partial_iB^{N+1}_n}t^{1-2s}\bar{u}_n\frac{\partial \bar{u}_n}{\partial\nu_z}dS_z.
\end{array}
\end{equation}

By \eqref{4.2}, noting that $|x-x_{0}|\leq
C\Lambda_{n}^{-\frac{1}{2}}$ for $x\in\partial_{i}B^N_{n}$, and $|z-z_{0}|\leq
C\Lambda_{n}^{-\frac{1}{2}}$ for $x\in\partial_{i}B^{N+1}_{n}$, we obtain
\begin{equation}\label{4.10}
\begin{array}{ll}
\text{~RHS~ of}\,\, \eqref{4.9}&\leq
C\Lambda_{n}^{-\frac{1}{2}}\ds\int_{\partial_{i} B^N_{n}}\Bigl(
u_{n}^{2}+|u_n|^{p_n}+\mu\frac{u_{n}^{2}}{|x|^{2s}}\Bigl)dS_x+C\int_{\partial_{i} B^{N+1}_{n}}t^{1-2s}|\nabla
u_{n}||u_{n}| dS_z \vspace{0.2cm}\\
&\quad+C\Lambda_{n}^{-\frac{1}{2}}\ds\int_{\partial_{i} B^{N+1}_{n}}t^{1-2s}|\nabla
u_{n}|^2 dS_z \vspace{0.2cm}\\
&\leq C\Lambda_{n}^{-\frac{1}{2}}\Lambda_{n}^{\frac{2s-1-N}{2}+\frac{N}{p_{1}}}
+C\Lambda_{n}^{1+\frac{2s-1-N}{2}+\frac{N}{p_{1}}}\vspace{0.2cm}\\
&\leq C\Lambda_{n}^{\frac{2s-N}{2}+\frac{N}{p_{1}}}.
\end{array}
\end{equation}

 Recalling
that in the proof of Lemma~\ref{lem2.5}, we have the decomposition
\[
u_n=u_0 +\sum_{j=1}^m
\rho_{0,\Lambda_{n,j}}(U_j)+\sum_{j=m+1}^h
\rho_{x_{n,j},\Lambda_{n,j}}(U_j)+\omega_n=:u_0+u_{n,1}+u_{n,2},
\]
 with $\|u_{n,2}\|\to 0$ as $n\to+\infty$.  By Proposition \ref{propbb3} and Lemma \ref{lemb.1}, we can verify that if $N>6s$,
\begin{equation}\label{e4.8}
\int_{\R^N} |U_j|^2dx<+\infty,\quad j=1,2,\cdots,h.
\end{equation}

On the other hand, let $B^N_{n,*}=B^N_{L\Lambda_{n}^{-1}}(x_{n})$, where
$L>0$ is so large that

\[
\int_{B^N_L(0)} |U_j|^2dx>0,\quad j=1,2,\cdots,h.
\]

Since $u_n=0$ in $\R^N\setminus \Omega$, we have
\begin{equation}\label{4.11}
\begin{split}
\int_{B^N_n}|u_n|^2\,dx&=\int_{B^N_{\ell_n\Lambda_n^{-\frac12}}(x_n)}|u_n|^2\,dx
\ge \int_{B^N_{n,*}}|u_n|^2\,dx\\
\qquad\qquad&\ge  \frac{1}{2}\int_{B^N_{n,*}}|u_{n,1}|^2dx-C\int_{B^N_{n,*}}
|u_0|^2dx-C\int_{B^N_{n,*}}|u_{n,2}|^2dx.
\end{split}
\end{equation}
Moreover, we have
\begin{equation}\label{4.12}
\begin{array}{ll}
\ds\int_{B^N_{n,*}}|u_{0}|^{2}dx&\leq
\Bigl(\ds\int_{B^N_{n,*}}|u_{0}|^{2^{*}_s}dx\Bigl)^{\frac{2}{2^{*}_s}}|B^N_{n,*}|^{1-\frac{2}{2^{*}_s}}
\leq
C\Lambda_{n}^{-2s}\|u_{0}\|^{2}_{L^{2^{*}_s}(B^N_{n,*})}=o(1)\Lambda_{n}^{-2s},
\end{array}
\end{equation}
and
\begin{equation}\label{4.13}
\begin{array}{ll}
\ds\int_{B^N_{n,*}}|u_{n,2}|^{2}dx\leq
C\Bigl(\ds\int_{B^N_{n,*}}|u_{n,2}|^{2^{*}_s}dx\Bigl)^{\frac{2}{2^{*}_s}}\Lambda_{n}^{-2s}=o(1)\Lambda_{n}^{-2s},
\end{array}
\end{equation}
since $\|u_{n,2}\|\rightarrow 0$ as $n\rightarrow\infty.$

On the other hand,  we may assume that
$\rho_{x_{n,1},\Lambda_{n,1}}(U_1)$ is the bubble with slowest
concentration rate. Then

\[
\int_{B^N_{n,*}}|u_{n,1}|^2dx\ge  \frac12
\int_{B^N_{n,*}}|\rho_{x_{n,1},\Lambda_{n,1}}(U_1)|^2dx+ O\Bigl(
 \sum_{j=2}^h
\int_{B^N_{n,*}}|\rho_{x_{n,j},\Lambda_{n,j}}(U_j)|^2dx \Bigr).
\]

By direct calculations, we can obtain
\[
\int_{B^N_{n,*}}|\rho_{x_{n,1},\Lambda_{n,1}}(U_1)|^2dx= \Lambda_{n,1}^{-2s}
\int_{B^N_L(0)} |U_1|^2dx\ge C' \Lambda_{n,1}^{-2s},
\]
for some constant $C'>0$.  Similarly, we have
\begin{equation}\label{1-5-8}
\int_{B^N_{n,*}}|\rho_{x_{n,j},\Lambda_{n,j}}(U_j)|^2dx= \Lambda_{n,j}^{-2s}
\int_{(B^N_{n,*})_{x_{n,j},\Lambda_{n,j}}}|U_j|^2dx,
\end{equation}
 where we use the notation
 $E_{x,\Lambda}=\{ y:  \Lambda^{-1}
y+ x\in E\}$  for any set  $E$.

If  $\frac{\Lambda_{n,j}}{\Lambda_{n,1}}\to +\infty$, then we obtain
from \eqref{1-5-8}
\[
\int_{B^N_{n,*}}|\rho_{x_{n,j},\Lambda_{n,j}}(U_j)|^2dx= o\bigl(
\Lambda_{n,1}^{-2s}\bigr).
\]

If  $\frac{\Lambda_{n,j}}{\Lambda_{n,1}}\le C<+\infty$, then
\[
\begin{split}
&(B^N_{n,*})_{x_{n,j},\Lambda_{n,j}} = \bigl\{ y:  \Lambda_{n,j}^{-1} y +
x_{n,j} \in B^N_{n,*}
\bigr\}\\
 = & \bigl\{ y:  |\Lambda_{n,j}^{-1} y +   x_{n,j}-x_{n,1}|\le L\Lambda_{n,1}^{-1}
\bigr\} \subset \bigl\{ y:  |y + \Lambda_{n,j} ( x_{n,j}-x_{n,1})|
\le C\bigr\}.
\end{split}
\]
Since  $ |\Lambda_{n,j} ( x_{n,j}-x_{n,1})|\to +\infty$  as $n\to
+\infty$, we find that $(B^N_{n,*})_{x_{n,j},\Lambda_{n,j}}$ moves to
infinity. Hence it follows from \eqref{e4.8} and \eqref{1-5-8} that

\[
\int_{B^N_{n,*}}|\rho_{x_{n,j},\Lambda_{n,j}}(U_j)|^2dx= o\bigl(
\Lambda_{n,1}^{-2s}\bigr).
\]

Therefore, we have proved that
 there exists a constant $C'>0$, such that
\begin{equation}\label{4.14}
\int_{B^N_{n,*}}|u_{n,1}|^2dx\ge C'\Lambda_n^{-2s}.
\end{equation}

Hence, from \eqref{4.11} to \eqref{4.14}, we get

\begin{equation}\label{4.15}
\text{~LHS~ of}\,\, \eqref{4.9}\geq \frac{C'}{4}\Lambda_{n}^{-2s}.
\end{equation}

Combing \eqref{4.10} and \eqref{4.15}, we obtain
\begin{equation}\label{4.16}
\Lambda_{n}^{-2s}\leq C \Lambda_{n}^{\frac{2s-N}{2}+\frac{N}{p_{1}}},
\end{equation}
where $p_{1}>2^{*}_s$ is any constant, satisfying $p_{1}<\frac{2^*_s\sqrt{\bar{\mu}}}{\sqrt{\bar{\mu}}-\sqrt{\bar{\mu}-\mu}}$. Choose
$p_{1}=\frac{2N}{N-6s}$ with $p_{1}+\delta<\frac{2^*_s\sqrt{\bar{\mu}}}{\sqrt{\bar{\mu}}-\sqrt{\bar{\mu}-\mu}}$, where $\delta>0$ is
a small constant. Then from the assumption on $\mu$, we see
$2s<\frac{N-2s}{2}-\frac{N}{p_1}$. So, we obtain a contradiction to \eqref{4.16}.
\end{proof}

\begin{proof}[\textbf{Proof of Theorem \ref{thm1.1}}]

 This is a direct consequence of Theorem \ref{thm1.2}. See for example \cite{cps,cs1,cs,yyy}.

\end{proof}

\appendix

\section{{\bf Some basic estimates on linear problems }}
In this section, we deduce some elementary estimates for solutions of linear elliptic problem involving Hardy potential.
These estimates are of independent interest.
\begin{lem}\label{lema.1-1}
Let $u\in H^s_0(\Omega)$ be a solution of \eqref{1.1}. Then one has
$$u\in L^p(\Omega),\ \ \forall p<\frac{2^*_s\sqrt{\bar{\mu}}}{\sqrt{\bar{\mu}}-\sqrt{\bar{\mu}-\mu}}.$$
\end{lem}
\begin{proof}
Just by the same argument as that of Lemma 2.1 in \cite{cp}, we can prove our result. So we omit it here.
\end{proof}

\begin{lem}\label{lema.1}
Let $w$ be a solution of
$$
\left\{%
\begin{array}{ll}
    (-\Delta)^s w-\ds\frac{\mu w}{|x|^{2s}} =a(x)v, & \hbox{$\text{in}~ \Omega$},\vspace{0.1cm} \\
   w=0,\,\, &\hbox{$\text{on}~\partial \Omega$}, \\
\end{array}%
\right.
$$
where $a(x)\geq 0,v\geq 0$ are functions and $a,v\in
C^{2s}(\Omega\backslash B_{\delta}(0))$ for any $\delta>0$ small. Then for any
$p>\frac{N}{N-2s}$ and $0\leq\mu<\bar{\mu}$ satisfying $p<\frac{2^*_s\sqrt{\bar{\mu}}}{\sqrt{\bar{\mu}}-\sqrt{\bar{\mu}-\mu}},$ there is a constant $C=C(p)$
such that
$$
||w||_{*,p}\leq
C||a||_{\frac{N}{2s}}||v||_{p}.
$$
\end{lem}

\begin{proof}
Let $q=\frac{p}{2^*_s}.$ Then $q>\frac{1}{2}.$

First we assume $p\geq 2^*_s.$ In this case $q\geq1.$
Let $\bar{\varphi}=\bar{w}\bar{w}_{L}^{2(q-1)},$ where $\bar{w}_{L}=\min\{\bar{w},L\}.$ Then we have
 $$
 \nabla
 \bar{\varphi}=2(q-1)\bar{w}_{L}^{2(q-1)}\nabla \bar{w}_{L}+\bar{w}_{L}^{2(q-1)}\nabla \bar{w}.
$$
Since $q>1$,
it is easy to see that $\nabla\bar{\varphi}\in L^{2}(t^{1-2s},\mathcal{D})$. Thus $\bar{\varphi}\in H^{1}_{0}(t^{1-2s},\mathcal{D})$. So, we have
\begin{equation}\label{a.1}
\int_{\mathcal{D}}t^{1-2s}\nabla\bar{w}\nabla(\bar{w}\bar{w}_{L}^{2(q-1)})dxdt
=\int_{\Omega}\frac{\mu w^2w_{L}^{2(q-1)}}{|x|^{2s}}dx+\int_{\Omega}a(x)vww_{L}^{2(q-1)}dx
\end{equation}
Letting $\bar{\eta}=\bar{w}\bar{w}_{L}^{q-1} $, from Hardy-Sobolev inequality we find
\begin{equation}\label{ha.2}
\int_{\Omega}\frac{\mu w^2w_{L}^{2(q-1)}}{|x|^{2s}}dx
\leq\frac{\mu}{\bar{\mu}}\int_{\Omega}\big|(-\Delta)^{\frac{s}{2}}(w w_{L}^{(q-1)})\big|^{2}dx=
\frac{\mu}{\bar{\mu}}\int_{\Omega}|(-\Delta)^{\frac{s}{2}}\eta|^{2}dx.
\end{equation}
Moreover, it follows from $|\nabla \bar{w}_{L}|\leq |\nabla \bar{w}|$ that
\begin{equation}\label{ha.3}
\begin{array}{ll}
\ds\int_{\Omega}|(-\Delta)^{\frac{s}{2}}\eta|^{2}dx
&=\ds\int_{\mathcal{D}}t^{1-2s}|\nabla \bar{\eta}|^2dxdt
=\ds\int_{\mathcal{D}}t^{1-2s}\big(\bar{w}_{L}^{2(q-1)}|\nabla \bar{w}|^{2}+(q^{2}-1)\bar{w}_{L}^{2(q-1)}|\nabla \bar{w}_{L}|^{2} \big)dxdt\vspace{0.2cm}\\
&\leq\ds\int_{\mathcal{D}}t^{1-2s}\Big[(q^{2}-\ds\frac{q^{2}(2q-2)}{2q-1})\bar{w}_{L}^{2(q-1)}|\nabla \bar{w}|^{2}
+\ds\frac{q^{2}(2q-2)}{2q-1}\bar{w}_{L}^{2(q-1)}|\nabla \bar{w}_{L}|^{2}\Big]dxdt\vspace{0.2cm}\\
&\leq\ds\frac{q^{2}}{2q-1}\ds\int_{\mathcal{D}}t^{1-2s}\big[\bar{w}_{L}^{2(q-1)}|\nabla \bar{w}|^{2}+(2q-2)\bar{w}_{L}^{2(q-1)}|\nabla \bar{w}_{L}|^{2}\big]dxdt\vspace{0.2cm}\\
&=\ds\frac{q^{2}}{2q-1}\ds\int_{\mathcal{D}}t^{1-2s}\nabla \bar{w}\nabla(\bar{w}\bar{w}_{L}^{2(q-1)})dxdt.
\end{array}
\end{equation}
From \eqref{a.1}-\eqref{ha.3}, we get
\begin{equation}\label{ha.4}
\Big(\frac{2q-1}{q^{2}}-\frac{\mu}{\bar{\mu}}\Big)\int_{\Omega}|(-\Delta)^{\frac{s}{2}}\eta|^{2}dx
\leq\int_{\Omega}a(x)v\varphi dx.
\end{equation}
Noting that $q<\frac{\sqrt{\bar{\mu}}}{\sqrt{\bar{\mu}}-\sqrt{\bar{\mu}-\mu}}$ implies
 $\frac{2q-1}{q^{2}}-\frac{\mu}{\bar{\mu}}\geq c_{0}>0,$
we obtain from \eqref{ha.4} that there is a $c^{'}>0$ such that
\begin{equation}\label{a.2}
\int_{\Omega}a(x)v\varphi dx\geq c_{0}\int_{\Omega}|(-\Delta)^{\frac{s}{2}}\eta|^{2}dx
\geq c'\Big(\int_{\Omega}|\eta|^{2^*_s}dx\Big)^{\frac{2}{2^*_s}}.
\end{equation}

On the other hand, by H\"{o}lder inequality we have
\begin{equation}\label{a.3}
\begin{array}{ll}
\ds\int_{\Omega}a(x)v\varphi dx&\leq \Big(\ds\int_{\Omega}|v|^pdx\Big)^{\frac{1}{p}}\Big(\ds\int_{\Omega}|a(x)\varphi|^{\frac{2^*_sq}{2^*_sq-1}}dx\Big)^{\frac{2^*_sq-1}{2^*_sq}}\vspace{0.2cm}\\
&\leq \|v\|_p\|a\|_{\frac{N}{2s}}\Big(\ds\int_{\Omega}|\varphi|^{\frac{2^*_sq}{2q-1}}dx\Big)^{\frac{2q-1}{2^*_sq}}\vspace{0.2cm}\\
&\leq \|v\|_p\|a\|_{\frac{N}{2s}}\Big(\ds\int_{\Omega}|\eta|^{2^*_s}dx\Big)^{\frac{2q-1}{2^*_sq}},\vspace{0.2cm}\\
\end{array}
\end{equation}
since $\frac{q}{2q-1}\leq 1.$

Thus,
\begin{equation}\label{ha.7}
c_{0}\Big(\ds\int_{\Omega}|\eta|^{2^*_s}dx\Big)^{\frac{1}{2^{*}_sq}}
\leq  \|v\|_p\|a\|_{\frac{N}{2s}}.
\end{equation}
From \eqref{ha.2}, \eqref{ha.4}, \eqref{a.3} and \eqref{ha.7}, we obtain
\begin{equation}\label{ha.8}
\begin{array}{ll}
\ds\int_{\Omega}\ds\frac{\mu}{|x|^{2s}}w^{2}w_{L}^{2(q-1)}dx\leq C\big( \|v\|_p\|a\|_{\frac{N}{2s}}\big)^{2q}.
\end{array}
\end{equation}
Letting $L\rightarrow\infty$ in \eqref{ha.7} and \eqref{ha.8}, we obtain the result.

Now we consider the case $q\in(\frac{1}{2},1).$ In this case,
$ww_{L}^{2(q-1)}$ may not be in $H^{s}_{0}(\Omega)$. Hence we
have to deal with it differently.

By the comparison principle, we know that $w\geq 0$ in $\Omega.$ For
any $\theta>0$ being a small number, let
$\bar{\eta}=(\bar{w}+\theta)^{2q-1}\bar{\xi}^{2},$  where $\bar{\xi}\geq0$
is a function satisfying $\bar{\xi}=0$ on $\partial\Omega\times[0,\infty)$; %$\bar{\xi}>0$ in $\mathcal{D}$;
$\bar{\xi}=1$ in
$\mathcal{D}_\theta:=\Omega_{\theta}\times[0,1)=\{x:x\in\Omega,d(x,\partial\Omega)\geq\theta^s\}\times[0,1)$; $0<\bar{\xi}<1$ on
$\Omega\setminus\Omega_\theta\times[0,1)$; $\bar{\xi}=0$ in
$\Omega\times[1,\infty)$
and $|\nabla\bar{\xi}|\leq\frac{2}{\theta^s}$. Then $\bar{\eta}\in H^{1}_{0}(t^{1-2s},\mathcal{D})$ and
$$
\nabla\bar{\eta}=(\bar{w}+\theta)^{2q-1}\nabla
\bar{\xi}^{2}+(2q-1)(\bar{w}+\theta)^{2(q-1)}\bar{\xi}^{2}\nabla \bar{w}.
$$
Moreover, from the assumption on $\bar{\xi}$, $\xi$ satisfies $\xi\geq0$, $\xi=0$ on $\partial \Omega$,
$\xi>0$ in $\Omega$ and $\xi=1$ in $\Omega_\theta$ and $|\nabla\xi|\leq\frac{2}{\theta^s}$.
So, we have
\begin{equation}\label{a.5}
\int_{\mathcal{D}}t^{1-2s}\nabla \bar{w} \nabla\bar{\eta}
dxdt=\int_{\Om}\mu\frac{w(w+\theta)^{2q-1}\xi^{2}}{|x|^{2s}}dx+\int_{\Omega}a(x)v(w+\theta)^{2q-1}\xi^{2}dx.
\end{equation}
On the other hand,
\begin{equation}\label{ha.10}
\begin{array}{ll}
&\ds\int_{\mathcal{D}}t^{1-2s}\nabla \bar{w} \nabla\bar{\eta}dxdt\vspace{0.2cm}\\
&=(2q-1)\ds\int_{\mathcal{D}}t^{1-2s}\bar{\xi}^{2}(\bar{w}+\theta)^{2(q-1)}|\nabla(\bar{w}+\theta)|^{2}dxdt
+\ds\int_{\mathcal{D}}t^{1-2s}(\bar{w}+\theta)^{2q-1}\nabla(\bar{w}+\theta)\nabla\bar{\xi}^{2}dxdt\vspace{0.2cm}\\
&=\ds\frac{2q-1}{q^{2}}\ds\int_{\mathcal{D}}t^{1-2s}\bar{\xi}^{2}|\nabla(\bar{w}+\theta)^{q}|^{2}dxdt
+\ds\int_{\mathcal{D}}t^{1-2s}(\bar{w}+\theta)^{2q-1}\nabla(\bar{w}+\theta)\nabla\bar{\xi}
^{2}dxdt\vspace{0.2cm}\\
&=\ds\frac{2q-1}{q^{2}}\ds\int_{\mathcal{D}}t^{1-2s}|\nabla(\bar{\xi}(\bar{w}+\theta))^{q}|^{2}dxdt
-\ds\frac{2(2q-1)}{q^{2}}\ds\int_{\mathcal{D}}t^{1-2s}q(\bar{w}+\theta)^{2q-1}\bar{\xi}\nabla\bar{\xi}\nabla(\bar{w}+\theta)dxdt
\vspace{0.2cm}\\
&\,\,\,\,\,-\ds\frac{2q-1}{q^{2}}\ds\int_{\mathcal{D}}t^{1-2s}(\bar{w}+\theta)^{2q}|\nabla\bar{\xi}|^{2}dxdt
+\ds\ds\int_{\mathcal{D}}t^{1-2s}(\bar{w}+\theta)^{2q-1}\nabla(\bar{w}+\theta)\nabla \bar{\xi}^{2}dxdt.
\end{array}
\end{equation}

From $a,v\in
C^{2s}(\Omega\backslash B_{\delta}(0))$ for any $\delta>0$ small,
it follows from \cite{ro} that $w\in C^{\beta}(\Omega\backslash B_{\delta}(0))$ for any $\beta\in[s,1+2s)$ and

$$w(x)\leq Cd^s(x,\partial\Omega)\leq C\theta^s, |\nabla w|\leq C, \forall x\in \Omega\backslash \Omega_\theta.$$

As a consequence, \eqref{ha.10} becomes
\begin{equation}\label{ha.11}
\begin{array}{ll}
&\ds\int_{\mathcal{D}}t^{1-2s}\nabla \bar{w} \nabla\bar{\eta}dxdt\vspace{0.2cm}\\
&=\ds\frac{2q-1}{q^{2}}\ds\int_{\mathcal{D}}t^{1-2s}|\nabla(\bar{\xi}(\bar{w}+\theta))^{q}|^{2}dxdt
+O\Big(\ds\int_{(\Omega\setminus \Omega_{\theta})\times [0,1)}t^{1-2s}(\theta^s)^{2q-1}dxdt\Big)\vspace{0.2cm}\\
&=\ds\frac{2q-1}{q^{2}}\ds\int_{\Omega}|(-\Delta)^{\frac{s}{2}}(\xi(w+\theta))^{q}|^{2}dxdt
+O\big((\theta^s)^{2q-1}\big).
\end{array}
\end{equation}

By \eqref{a.5} and \eqref{ha.11}, we get
\begin{equation}\label{ha.12}
\begin{array}{ll}
&\ds\frac{2q-1}{q^{2}}\ds\int_{\Omega}|(-\Delta)^{\frac{s}{2}}(\xi(w+\theta))^{q}|^{2}dx
+O\big((\theta^s)^{2q-1}\big)-\ds\int_{\Omega}\mu\frac{w(w+\theta)^{2q-1}\xi^{2}}{|x|^{2s}}dx\vspace{0.2cm}\\
&=\ds\int_{\Omega}a(x)v(w+\theta)^{2q-1}\xi^{2}dx.
\end{array}
\end{equation}
But,
\begin{equation}\label{ha.13}
\begin{array}{ll}
\mu\ds\int_{\Omega}\frac{w(w+\theta)^{2q-1}\xi^{2}}{|x|^{2s}}dx
\leq \mu\ds\int_{\Omega}\frac{(w+\theta)^{2q}\xi^{2}}{|x|^{2s}}
dx \leq \frac{\mu}{\bar{\mu}} \ds\int_{\Omega}|(-\Delta)^{\frac{s}{2}}(\xi(w+\theta)^{q})|^{2}dx.
\end{array}
\end{equation}
From the assumptions on $q$ and $\mu$, \eqref{ha.12} and \eqref{ha.13}, we can deduce
\begin{equation}\label{ha.14}
\begin{array}{ll}
C'\Big(\ds\int_{\Omega}(\xi(w+\theta)^{q})^{2^{*}_s}dx\Big)^{\frac{2}{2^{*}_s}}
+O\big((\theta^s)^{2q-1}\big)
\leq\ds\int_{\Omega}a(x)v(w+\theta)^{2q-1}\xi^{2}dx,
\end{array}
\end{equation}
and
\begin{equation}\label{ha.15}
\begin{array}{ll}
C'\ds\int_{\Omega}\mu\frac{(w+\theta)^{2q}\xi^{2}}{|x|^{2s}}
dx
+O\big((\theta^s)^{2q-1}\big)
\leq\ds\int_{\Omega}a(x)v(w+\theta)^{2q-1}\xi^{2}dx.
\end{array}
\end{equation}
Letting $\theta\rightarrow0$ in \eqref{ha.14} and \eqref{ha.15}, we find
\begin{eqnarray*}
C'\Big(\ds\int_{\Omega}w^{q2^{*}_s}dx \Big)^{\frac{2}{2^{*}_s}}
\leq\ds\int_{\Omega}a(x)vw^{2q-1}dx
\leq ||a||_{\frac{N}{2s}}||v||_{p}||w||^{2q-1}_{p},
\end{eqnarray*}
and
\begin{eqnarray*}
C'\ds\int_{\Omega}\mu\frac{w^{2q}}{|x|^{2s}}
dx
\leq\ds\int_{\Omega}a(x)vw^{2q-1}dx
\leq ||a||_{\frac{N}{2s}}||v||_{p}||w||^{2q-1}_{p}.
\end{eqnarray*}
Therefore, the result follows.
\end{proof}

\begin{lem}\label{lema.2}
Let $w$ be a solution of
$$
\left\{%
\begin{array}{ll}
    (-\Delta)^s w-\ds\frac{\mu w}{|x|^{2s}}=f(x), & \hbox{$\text{in}~ \Omega$},\vspace{0.1cm} \\
   w=0,\,\, &\hbox{$\text{on}~\partial \Omega$}. \\
\end{array}%
\right.
$$

Suppose that $f\in C^{s}(\Omega\backslash B_{\delta}(0))$ for any small $\delta>0$.
Then for any
$\frac{N}{2s}>p\geq1$ and $\mu$ with $\frac{Np}{N-2sp}<\frac{2^*_s\sqrt{\bar{\mu}}}{\sqrt{\bar{\mu}}-\sqrt{\bar{\mu}-\mu}},$
there is a constant
$C=C(p)$ such that
$$
||w||_{*,\frac{Np}{N-2sp}}\leq C||f||_p.
$$
\end{lem}
\begin{proof}
Similar to the proof of Lemma \ref{lema.1}, we can deduce if
$q>\frac{1}{2}$ with $q<\frac{\sqrt{\bar{\mu}}}{\sqrt{\bar{\mu}}-\sqrt{\bar{\mu}-\mu}},$ then
\begin{eqnarray*}
C'||w||^{2q}_{2^{*}_sq}\leq\ds\int_{\Omega}f(x)w^{2q-1}dx,
\end{eqnarray*}
and
\begin{eqnarray*}
C'\int_{\Omega}\frac{\mu|w|^{2q}}{|x|^{2s}}dx\leq\ds\int_{\Omega}f(x)w^{2q-1}dx.
\end{eqnarray*}
For any $\frac{N}{2s}>p\geq1,$ see $2p-2^{*}_s(p-1)>0.$ Let $q=\frac{p}{2p-2^{*}_s(p-1)}>\frac{1}{2}.$
Then
\begin{equation}\label{a.7}
2^{*}_sq=\frac{2^{*}_s p}{2p-2^{*}_s(p-1)}=\frac{Np}{N-2sp}
\,\,\,\text{and}\,\,\,\,(2q-1)\frac{p}{p-1}=2^{*}_sq.
\end{equation}
So,
\begin{eqnarray*}
\Big|\ds\int_{\Omega}f(x)w^{2q-1}dx\Big|\leq ||f||_{p}
\Big(\int_{\Omega}\|w\|^{\frac{(2q-1)p}{p-1}}dx\Big)^{1-\frac{1}{p}}
\leq ||f||_{p}\|w\|_{2^{*}_sq}^{2q-1}.
\end{eqnarray*}
As a result,
$$
\|w\|_{2^{*}_sq}\leq C ||f||_{p},
$$
and
$$
\mu\int_{\Omega}\frac{|w|^{2q}}{|x|^{2s}}dx\leq C ||f||_{p}^{2q}.
$$

So the result follows.
\end{proof}

\begin{lem}\label{lema.3}
Let $w\geq 0$ be a solution of
$$
\left\{%
\begin{array}{ll}
    (-\Delta)^s w-\ds\frac{\mu w}{|x|^{2s}}=a(x)v, & \hbox{$\text{in}~ \Omega$},\vspace{0.1cm} \\
   w=0,\,\, &\hbox{$\text{on}~\partial \Omega$}, \\
\end{array}%
\right.
$$
where $a(x)\geq0$ and $v\geq0$ are functions satisfying $a,v\in
C^{2s}(\Omega\backslash B_{\delta}(0))$ for any small $\delta>0$. Then for any $2^{*}_s>p_{2}>\frac{N}{N-2s},$
there is a constant $C=C(p_{2})$ such that
$$
||w||_{*,p_{2}}\leq C||a||_{r}||v||_{2^{*}_s},
$$
where $r$ is determined by
$\frac{1}{r}=\frac{1}{p_{2}}+\frac{2s}{N}-\frac{1}{2^{*}_s}.$
\end{lem}

\begin{proof}
Let $q=\frac{2p_{2}}{2^{*}_s}.$ Since $p_{2}\in\big(\frac{N}{N-2s},2^{*}_s\big),$ we
see $\frac{q}{2}\in(\frac{1}{2},1).$  Let $t=\frac{2N}{N+2s}$. Similar to the proof of Lemma \ref{lema.1}, we obtain
\begin{equation}\label{a.8}
\begin{array}{ll}
C||w||^{q}_{p_{2}}&=C||w||^{q}_{\frac{2^{*}_sq}{2}}
\leq\ds\int_{\Omega}a(x)vw^{q-1}dx
\leq ||v||_{2^{*}_s}||a||_{r}\Big(\ds\int_{\Omega}|w|^{\frac{(q-1)tr}{r-t}}dx\Big)^{\frac{r-t}{rt}},
\end{array}
\end{equation}
and
$$C\int_{\Omega}\frac{\mu|w|^{q}}{|x|^{2s}}dx\leq ||v||_{2^{*}_s}||a||_{r}\Big(\ds\int_{\Omega}|w|^{\frac{(q-1)tr}{r-t}}dx\Big)^{\frac{r-t}{rt}}.$$
By the definition, choose
$$
\frac{(q-1)tr}{r-t}=\frac{\frac{2p_{2}}{2^{*}_s}-1}{\frac{1}{t}-\frac{1}{r}}=p_{2}.
$$
So,
$$
||w||^{q}_{p_{2}}\leq C||v||_{2^{*}_s}
||a||_{r}||w||_{p_{2}}^{q-1}.
$$
Moreover, it is easy to check
$$
\frac{1}{p_{2}}=\frac{1}{r}+\frac{1}{2^{*}_s}-\frac{2s}{N}.
$$

Therefore, the result follows.
\end{proof}

\begin{lem}\label{lema.4}
Let $w\geq 0$ be a weak solution of
$$
    (-\Delta)^s w-\ds\frac{\mu w}{|x|^{2s}}=a(x)w\,\,\,\, \text{in}~ \R^{N},
$$
where $a(x)\geq0$. Suppose that there is a small
constant $\delta>0$ such that
$\int_{B_{1}(\bar{x})}|a|^{\frac{N}{2s}}dx\leq\delta,$
then, for any $p>\max \{2^{*}_s,2^{\sharp}\}$ and $0\leq\mu<\bar{\mu}$ satisfying $p<\min\{\frac{2^*_s\sqrt{\bar{\mu}}}{\sqrt{\bar{\mu}}-\sqrt{\bar{\mu}-\mu}},
\frac{2^{\sharp}\sqrt{\bar{\mu}}}{\sqrt{\bar{\mu}}-\sqrt{\bar{\mu}-\mu}}\}$ there is a constant $C=C(p)$ such that
$$
||\bar{w}||_{L^p(t^{1-2s},B^{N+1}_{\frac{1}{2}}(\bar{x}))}+
||w||_{L^p(B^N_{\frac{1}{2}}(\bar{x}))}\leq C||\bar{w}||_{L^\gamma(t^{1-2s},B^{N+1}_{1}(\bar{x}))},
$$
where $\gamma<\min\{2^*_s,2^{\sharp}\}$ and $2^{\sharp}=\frac{2(N+1)}{N}$.
\end{lem}

\begin{proof}
Let $1\geq R>r>0$.
Take $\bar{\xi}\in C^{2}_{0}(B^{N+1}_R(\bar{x})),$ with
$\bar{\xi}=1$ in $B^{N+1}_r(\bar{x})$, $0\leq\bar{\xi}\leq1,$ and
$|\nabla\bar{\xi}|\leq\frac{C}{R-r}.$ Let
$q=\frac{p}{2^*_s}, \bar{\eta}=\bar{\xi}^{2}\bar{w}\bar{w}_{L}^{2(q-1)}.$ We have
$$
\int_{\R^{N+1}_{+}}t^{1-2s}\nabla \bar{w}\nabla \bar{\eta}dxdt -\int_{\R^{N}}\ds\frac{\mu }{|x|^{2s}}w\eta dx\leq
\int_{\R^{N}}a(x)w\eta dx.
$$

Firstly, by H\"{o}lder inequality, we have
\begin{equation}\label{a.9}
\begin{array}{ll}
\ds\int_{\R^{N}}a(x)w\eta dx&\leq
\Bigl(\ds\int_{B_{1}(\bar{x})}|a(x)|^{\frac{N}{2s}}dx\Bigl)^{\frac{2s}{N}}
\Bigl(\ds\int_{B_{R}(\bar{x})}(\xi
ww_{L}^{q-1})^{2^{*}_s}dx\Bigl)^{\frac{2}{2^{*}_s}}\vspace{0.2cm}\\
&\leq C\delta\ds\int_{\R^{N+1}_{+}}t^{1-2s}|\nabla(\bar{\xi} \bar{w}\bar {w}_{L}^{q-1})|^2dxdt.
\end{array}
\end{equation}

On the other hand, similar to the proof of Lemma \ref{lema.1}, by Hardy-Sobolev inequality we can deduce that if $q<\frac{\sqrt{\bar{\mu}}}{\sqrt{\bar{\mu}}-\sqrt{\bar{\mu}-\mu}},$ then
\begin{equation}\label{a.10}
\begin{array}{ll}
&\ds\int_{\R^{N+1}_{+}}t^{1-2s}\nabla \bar{w}\nabla \bar{\eta}dxdt -\ds\int_{\R^{N}}\ds\frac{\mu }{|x|^{2s}}w\eta dx\vspace{0.2cm}\\
&\geq C'\ds\int_{\R^{N+1}_{+}}t^{1-2s}|\nabla(\bar{\xi}\bar{w}\bar {w}_{L}^{q-1})|^2dxdt
-C\int_{\R^{N+1}_{+}}t^{1-2s}|\nabla\bar{\xi}|^{2}(\bar{w}\bar {w}_{L}^{q-1})^{2}dxdt.
\end{array}
\end{equation}

If $\delta$ is small enough,
it follows from \eqref{a.9} and \eqref{a.10} that there exists $C>0$
such that
\begin{equation}\label{a.11}
\begin{array}{ll}
\ds\int_{\R^{N+1}_{+}}t^{1-2s}|\nabla(\bar{\xi} \bar{w}\bar {w}_{L}^{q-1})|^2dxdt
\leq C\int_{\R^{N+1}_{+}}t^{1-2s}|\nabla\bar{\xi}|^{2}(\bar{w}\bar {w}_{L}^{q-1})^{2}dxdt.
\end{array}
\end{equation}

Using the Sobolev inequality and Lemma \ref{lem3.1.1}, we obtain from \eqref{a.11} that
\begin{equation}\label{a.12}
\begin{array}{ll}
&\Bigl(\ds\int_{\R^N}(\xi w w_{L} ^{q-1})^{2^{*}_s}dx\Bigl)^{\frac{2}{2^{*}_s}}
+\Bigl(\ds\int_{\R^{N+1}_{+}}t^{1-2s}(\bar{\xi} \bar{w} \bar{w}_{L} ^{q-1})^{\frac{2(N+1)}{N}}dxdt\Bigl)^{\frac{N}{N+1}}\vspace{0.2cm}\\
&\leq C\ds\int_{\R^{N+1}_{+}}t^{1-2s}|\nabla(\bar{\xi}\bar{w}\bar {w}_{L}^{q-1})|^2dxdt
\leq C\ds\int_{\R^{N+1}_{+}}t^{1-2s}|\nabla\bar{\xi}|^{2}(\bar{w}\bar {w}_{L}^{q-1})^{2}dxdt,
\end{array}
\end{equation}
which yields
\begin{equation}\label{a.13}
\begin{array}{ll}
\Bigl(\ds\int_{B_{r}(\bar{x})}|w|^{q2^{*}_s}dx\Bigl)^{\frac{1}{q2^{*}_s}} \leq
\Bigl(\frac{C}{R-r}\Bigl)^{\frac{1}{q}}\Bigl(\ds\int_{B^{N+1}_R(\bar{x})\backslash B^{N+1}_r(\bar{x})}t^{1-2s}
|\bar{w}|^{2q}dxdt\Bigl)^{\frac{1}{2q}},
\end{array}
\end{equation}
and
\begin{equation}\label{a.131}
\begin{array}{ll}
\Bigl(\ds\int_{B^{N+1}_r(\bar{x})}t^{1-2s}|\bar{w}|^{q2^{\sharp}}dxdt\Bigl)^{\frac{1}{q2^{\sharp}}} \leq
\Bigl(\frac{C}{R-r}\Bigl)^{\frac{1}{q}}\Bigl(\ds\int_{B^{N+1}_R(\bar{x})\backslash B^{N+1}_r(\bar{x})}t^{1-2s}
|\bar{w}|^{2q}dxdt\Bigl)^{\frac{1}{2q}},
\end{array}
\end{equation}
where $2^{\sharp}=\frac{2(N+1)}{N}$.

Let $\chi=\frac{2^{\sharp}}{2}=\frac{N+1}{N}>1.$ For any
$0<r^{*}<R^{*}<1,$ define
$r_{i}=r^{*}+\frac{1}{2^{i}}(R^{*}-r^{*}),i=0,1,2,\cdot\cdot\cdot.$
Then $r_{i}-r_{i+1}=\frac{1}{2^{i+1}}(R^{*}-r^{*}).$ Taking
$R=r_{i},r=r_{i+1},q=\chi^{i}$ in \eqref{a.131}, we get
\begin{equation}\label{a.14}
\begin{array}{ll}
\Bigl(\ds\int_{B^{N+1}_{r_{i+1}}(\bar{x})}t^{1-2s}|\bar{w}|^{2\chi^{i+1}}dxdt\Bigl)^{\frac{1}{2\chi^{i+1}}} \leq
\Bigl(\frac{C2^{i+1}}{R^{*}-r^{*}}\Bigl)^{\frac{1}{\chi^{i}}}\Bigl(\ds\int_{B^{N+1}_{r_i}(\bar{x})}t^{1-2s}
|\bar{w}|^{2\chi^{i}}dxdt\Bigl)^{\frac{1}{2\chi^{i}}}.
\end{array}
\end{equation}

By iteration, for any
$0<r^{*}<R^{*}<1,$ we can obtain from \eqref{a.14}
\begin{equation}\label{a.15}
\begin{array}{ll}
\Bigl(\ds\int_{B^{N+1}_{r_{i+1}}(\bar{x})}t^{1-2s}|\bar{w}|^{2\chi^{i+1}}dxdt\Bigl)^{\frac{1}{2\chi^{i+1}}} \leq
\frac{C}{(R^{*}-r^{*})^{\sum_{j=1}^{i}\frac{1}{\chi^{j}}}}\Bigl(\ds\int_{B^{N+1}_{R^*}(\bar{x})}t^{1-2s}
|w|^{2^{\sharp}}dxdt\Bigl)^{\frac{1}{2^{\sharp}}}.
\end{array}
\end{equation}

Note that
$\ds\sum_{j=1}^{i}\frac{1}{\chi^{j}}<\ds\sum_{j=1}^{\infty}\frac{1}{\chi^{j}}=\frac{\frac{1}{\chi}}{1-\frac{1}{\chi}}=N.$
Hence, we have proved that for any $p>2^{\sharp}$ satisfying $p<\frac{2^{\sharp}\sqrt{\bar{\mu}}}{\sqrt{\bar{\mu}}-\sqrt{\bar{\mu}-\mu}}$, there is a
$\sigma>0$ depending on $p$ such that
\begin{equation}\label{a.16}
||\bar{w}||_{L^p(t^{1-2s},B^{N+1}_{r}(\bar{x}))}\leq\frac{C}{(R-r)^\sigma} ||\bar{w}||_{L^{2^{\sharp}}(t^{1-2s},B^{N+1}_R(\bar{x}))},\,\,\,\,\,0<r<R\leq1.
\end{equation}

Applying Young's inequality, we have
\begin{equation}\label{a.17}
\begin{array}{ll}
&\ds\frac{C}{(R-r)^{\sigma}}\Bigl(\int_{B^{N+1}_R(\bar{x})}t^{1-2s}|\bar{w}|^{2^{\sharp}}dxdt\Bigl)^{\frac{1}{2^{\sharp}}}\vspace{0.2cm}\\
&\leq
\ds\frac{C}{(R-r)^{\sigma}}\Bigl(\int_{B^{N+1}_R(\bar{x})}t^{1-2s}|\bar{w}|^{\gamma}dxdt\Bigl)^{\frac{\kappa}{\gamma}}
\Bigl(\int_{B^{N+1}_R(\bar{x})}t^{1-2s}|\bar{w}|^{p}dxdt\Bigl)^{\frac{1-\kappa}{p}}\vspace{0.2cm}\\
&\leq\ds\frac{1}{2}||\bar{w}||_{L^p(t^{1-2s},B^{N+1}_{R}(\bar{x}))}+\frac{C}{(R-r)^{\frac{\sigma}{\kappa}}}||\bar{w}||_{L^\gamma(t^{1-2s},B^{N+1}_{R}(\bar{x}))},
\end{array}
\end{equation}
where $0<\kappa<1, \gamma<2^{\sharp}$ and $p>2^{\sharp}$ with $p<\frac{2^{\sharp}\sqrt{\bar{\mu}}}{\sqrt{\bar{\mu}}-\sqrt{\bar{\mu}-\mu}}$.

So,
\begin{equation}\label{a.18}
||\bar{w}||_{L^p(t^{1-2s},B^{N+1}_{r}(\bar{x}))}\leq
\ds\frac{1}{2}||\bar{w}||_{L^p(t^{1-2s},B^{N+1}_{R}(\bar{x}))}+\frac{C}{(R-r)^{\frac{\sigma}{\kappa}}}||\bar{w}||_{L^\gamma(t^{1-2s},B^{N+1}_{R}(\bar{x}))},
\end{equation}
where $\gamma<2^{\sharp}$ and $2^{\sharp}<p<\frac{2^{\sharp}\sqrt{\bar{\mu}}}{\sqrt{\bar{\mu}}-\sqrt{\bar{\mu}-\mu}}$.

By using iteration argument, we deduce from \eqref{a.18} that for any $p>2^{\sharp}$
$$||\bar{w}||_{L^p(t^{1-2s},B^{N+1}_{r}(\bar{x}))}\leq
\frac{C}{(R-r)^{\frac{\sigma}{\kappa}}}||\bar{w}||_{L^\gamma(t^{1-2s},B^{N+1}_{R}(\bar{x}))},$$
where $p>2^{\sharp}$ satisfies $p<\frac{2^{\sharp}\sqrt{\bar{\mu}}}{\sqrt{\bar{\mu}}-\sqrt{\bar{\mu}-\mu}}$
and $\gamma<2^{\sharp}$.

Similarly, by applying \eqref{a.13} and iteration argument, we can get that
$$||w||_{L^p(B^N_{\frac{1}{2}}(\bar{x}))}\leq C||\bar{w}||_{L^\gamma(t^{1-2s},B^{N+1}_{1}(\bar{x}))},$$
where $p>2^{*}_s$ satisfies $p<\frac{2^*_s\sqrt{\bar{\mu}}}{\sqrt{\bar{\mu}}-\sqrt{\bar{\mu}-\mu}}$
and $\gamma<2^{*}_s$.
This completes our proof.
\end{proof}

\section{{\bf A Decay Estimate}}
Let $u$ be a solution  of
\begin{equation}\label{b.1}
\begin{cases}
\displaystyle(-\Delta)^s u-\frac{\mu u}{|x|^{2s}}=
|u|^{2^*_s-2}u  &\text{in}\;\R^N,\\
u\in H^{s}(\R^N).
\end{cases}
\end{equation}
In this section, we will estimate the decay of the solution of
\eqref{b.1}. We have the following result:
\begin{lem}\label{lemb.1}
Let $u$ be a solution of \eqref{b.1}. Then there exists a constant $\beta\in(s,\frac{N-2s}{2})$ such that
$$
|u(x)|\leq\frac{C}{|x|^{\frac{N-2s}{2}+\beta}},\,\,\,\,\forall |x|\geq1.
$$

\end{lem}

First, similar to  Proposition B.1 in \cite{cps}, we can prove
\begin{lem}\label{lemb.1.1}
Let $u$ be a solution of \eqref{b.1}. Then there exists a constant $\tilde{\beta}>\frac{N-2s}{2}$ such that
$$
|u(x)|\leq\frac{C}{|x|^{N+\tilde{\beta}}},\,\,\,\,\forall |x|\geq1.
$$

\end{lem}

\begin{proof}
Choose $R_0>0$ large. For any $R>r>R_0$,
take $\bar{\xi}\in C_0^\infty(\R^{N+1}_+)$ with $\bar{\xi}=0$ in $B^{N+1}_r(0)$, $\bar{\xi}=1$ in $\R^{N+1}_+\backslash B^{N+1}_R(0)$,
$0\leq\bar{\xi}\leq1$ and $|\nabla \bar{\xi}|\leq \frac{C}{R-r}$. Let $\tilde{\eta}=\bar{\xi}^2\bar{u}^{1+2(q-1)}_+$.
Since for any small $\delta>0$,
$$
\ds\int_{\R^N\backslash B_{R_0}(0)}|u|^{(2^*_s-2)\frac{N}{2s}}dx
=\ds\int_{\R^N\backslash B_{R_0}(0)}|u|^{2^*_s}dx\leq \delta,
$$
if $R_0>0$ large enough, we can prove in a similar way as in \eqref{a.16} that for any
$2^{\sharp}<p<\frac{2^{\sharp}\sqrt{\bar{\mu}}}{\sqrt{\bar{\mu}}-\sqrt{\bar{\mu}-\mu}}$, there is an $R>0$ large depending on $p$
such that
\begin{equation}\label{b.1.1}
\|\bar{u}_{+}\|_{L^p(t^{1-2s},\R^{N+1}_+\setminus B^{N+1}_{2R}(0))}\leq\frac{C}{R^{N-\delta_p}}\|\bar{u}_{+}\|_{L^{2^{\sharp}}(t^{1-2s},\R^{N+1}_+\setminus B^{N+1}_R(0))},
\end{equation}
where $\delta_p>0$ depending on $p$ and $2^{\sharp}=\frac{2(N+1)}{N}$.

Next we estimate $\|\bar{u}_{+}\|_{L^{2^{\sharp}}(\R^{N+1}_+\setminus B^{N+1}_{R}(0))}$. Let $\bar{\eta}=\bar{\xi}^2\bar{u}$, $\bar{\xi}=0$ in $B^{N+1}_R(0)$,
$\bar{\xi}=1$ in $\R^{N+1}_+\backslash B^{N+1}_{2R}(0)$,
$0\leq\bar{\xi}\leq1$ and $|\nabla \bar{\xi}|\leq \frac{C}{R}$. Then similar to the proof of \eqref{a.12},
 by H\"{o}lder inequality
we have
\begin{equation}\label{b.1.3}
\begin{array}{ll}
\Big(\ds\int_{\R^{N+1}_+}t^{1-2s}|\xi \bar{u}_{+}|^{2^{\sharp}}dxdt\Big)^{\frac{1}{2^{\sharp}}}
&\leq \ds\frac{C}{R}\Big(\ds\int_{B^{N+1}_{2R}(0)\setminus B^{N+1}_R(0)}t^{1-2s}| \bar{u}_{+}|^{2}dxdt\Big)^{\frac{1}{2}}\vspace{0.2cm}\\
&\leq\frac{C}{R^{1-\frac{N+2-2s}{2(N+1)}}}\Big(\ds\int_{B^{N+1}_{2R}(0)\setminus B^{N+1}_{R}(0)}t^{1-2s}| \bar{u}_{+}|^{2^{\sharp}}dxdt\Big)^{\frac{1}{2^{\sharp}}}.
 \end{array}
\end{equation}
As a result,
\begin{equation}\label{b.1.4}
\begin{array}{ll}
~~~~&\ds\int_{\R^{N+1}_+\backslash B^{N+1}_{2R}(0)}t^{1-2s}| \bar{u}_{+}|^{2^{\sharp}}dxdt\vspace{0.2cm}\\
&\leq\Bigl(\ds\frac{C}{R^{A}}\Bigl)^{2^{\sharp}}\ds\int_{B^{N+1}_{2R}(0)\setminus B^{N+1}_{R}(0)}t^{1-2s}| \bar{u}_{+}|^{2^{\sharp}}dxdt\vspace{0.2cm}\\
&=\Bigl(\ds\frac{C}{R^{A}}\Bigl)^{2^{\sharp}}\ds\int_{\R^{N+1}_+\setminus B^{N+1}_R(0)}t^{1-2s}| \bar{u}_{+}|^{2^{\sharp}}dxdt
-\Bigl(\frac{C}{R^{A}}\Bigl)^{2^{\sharp}}\ds\int_{\R^{N+1}_+\setminus B^{N+1}_{2R}(0)}t^{1-2s}| \bar{u}_{+}|^{2^{\sharp}}dxdt,
\end{array}
\end{equation}
where $A=1-\frac{N+2-2s}{2(N+1)}$.
So,
\begin{equation}\label{b.1.5}
\int_{\R^{N+1}_+\setminus B^{N+1}_{2R}(0)}t^{1-2s}| \bar{u}_{+}|^{2^{\sharp}}dxdt\leq
\frac{C'}{1+C'}\ds\int_{\R^{N+1}_+\setminus B^{N+1}_R(0)}t^{1-2s}| \bar{u}_{+}|^{2^{\sharp}}dxdt,
\end{equation}
where
\begin{equation}\label{b.1.6}
C'=\Bigl(\frac{C}{R^{A}}\Bigl)^{2^{\sharp}}.
\end{equation}
Let $\Psi(R)=\int_{\R^{N+1}_+\setminus B^{N+1}_{R}(0)}t^{1-2s}| u_{+}|^{2^{\sharp}}dxdt$ and $\tau=\frac{C'}{1+C'}$. Then from \eqref{b.1.5},
$$\Psi(2R)\leq \tau\Psi(R),\,\,\,\forall R\geq R_0,$$
which implies that
$$\Psi(2^iR_0)\leq \tau^i\Psi(R_0).$$

For any $|(x,t)|\geq R_0$, there is an $i$ such that
$$2^iR_0\leq |(x,t)|\leq 2^{i+1}R_0.$$
Hence
$$\Psi(|(x,t)|)\leq \Psi(2^iR_0)\leq \tau^i\Psi(R_0)\leq \tau^{\log_2|(x,t)|-\log_2(2R_0)}\Psi(R_0).$$
Since
$$\tau^{\log_2|(x,t)|}=2^{\log_2|(x,t)|\log_2\tau}=|(x,t)|^{\log_2\tau},$$
we have
$$\Psi(|(x,t)|)\leq C|(x,t)|^{\log_2\tau}=\frac{C}{|(x,t)|^{\log_2\frac{1}{\tau}}}.$$
So we have proved that there is a $\sigma>0$ independent of $p$ such that
$$\Psi(|(x,t)|)\leq \frac{C}{|(x,t)|^{\sigma}},\,\,\,\,|(x,t)|\geq R_0,$$
where $\sigma=\log_2\frac{1}{\tau}$.
Fix $p>2^{\sharp}$ and $p<\frac{2^{\sharp}\sqrt{\bar{\mu}}}{\sqrt{\bar{\mu}}-\sqrt{\bar{\mu}-\mu}}$. It follows from \eqref{b.1.1} that
$$\|\bar{u}_{+}\|_{L^p(t^{1-2s},\R^{N+1}_+\setminus B^{N+1}_{2R}(0))}\leq\frac{C}{R^{N-\delta_p+\frac{\sigma}{2^{\sharp}}}}
,\,\,\,|(x,t)|=R\geq R_0.$$
Using the definition of $\tau$, we can choose $R_0$ large enough such that $\frac{\sigma}{2^{\sharp}}-\delta_p\geq \tilde{\beta}>\frac{N-2s}{2}$ and
then
\begin{equation}\label{b.1.7}
\|\bar{u}_{+}\|_{L^p(t^{1-2s},\R^{N+1}_+\setminus B^{N+1}_{2R}(0))}\leq\frac{C}{R^{N+\tilde{\beta}}},\,\,\,|x|=R\geq R_0.
\end{equation}
Similarly,
\begin{equation}\label{b.1.8}
\|(-\bar{u})_{+}\|_{L^p(t^{1-2s},\R^{N+1}_+\setminus B^{N+1}_{2R}(0))}\leq\frac{C}{R^{N+\tilde{\beta}}},\,\,\,|x|=R\geq R_0.
\end{equation}
Thus, we obtain
$$
\|\bar{u}\|_{L^p(t^{1-2s},\R^{N+1}_+\setminus B^{N+1}_{2R}(0))}\leq\frac{C}{R^{N+\tilde{\beta}}},\,\,\,|x|=R\geq R_0.
$$

Now for any $(x,t)$ with $|(x,t)|=4R>2R_0$,
$$|\bar{u}(x,t)|\leq \max\limits_{(y,\tilde{t})\in B^{N+1}_1(x,t)}|\bar{u}(y,\tilde{t})|\leq C\|\bar{u}\|_{L^p(t^{1-2s}, B^{N+1}_{2}(x,t))}
\leq C\|\bar{u}\|_{L^p(t^{1-2s},\R^{N+1}_+\setminus B^{N+1}_{2R}(0))}\leq\frac{C}{R^{N+\tilde{\beta}}}.$$
By the fact that $u(x)=\bar{u}(x,0)$, the result follows.
\end{proof}

\begin{proof}[\textbf{Proof of Lemma \ref{lemb.1}}]

For every $\alpha\in(-\frac{N}{2}-s,\frac{N}{2}-s)$, let $\vartheta_\alpha=|x|^{\frac{2s-N}{2}+\alpha}$. Then it follows from
Lemma 3.1 in \cite{fa} that
$$(-\Delta)^s\vartheta_\alpha=\Upsilon_\alpha|x|^{-2s}\vartheta_\alpha\,\,\,\,\,\,\,\text{in}\,\R^N\setminus\{0\},$$
where
$$\Upsilon_\alpha=
2^{2s}\frac{\Gamma(\frac{N+2s+2\alpha}{4})\Gamma(\frac{N+2s-2\alpha}{4})}{\Gamma(\frac{N-2s-2\alpha}{4})\Gamma(\frac{N-2s+2\alpha}{4})}.$$
So by the assumption $0\leq\mu<\Upsilon_s$ and $\Upsilon_\alpha$ is even on $\alpha$, we can find $\bar{\alpha}\in(-\frac{N-2s}{2},-s)$
such that for $|x|\geq1$,
$$(-\Delta)^s\vartheta_{\bar{\alpha}}-\mu|x|^{-2s}\vartheta_{\bar{\alpha}}
=(\Upsilon_{\bar{\alpha}}-\mu)|x|^{-\frac{2s+N}{2}+\bar{\alpha}}
\geq \frac{C}{|x|^{(N+\tilde{\beta})(2^*_s-1)}}.
$$

On the other hand, from Lemma \ref{lemb.1.1},
$$(-\Delta)^su_+-\mu|x|^{-2s}u_+\leq u_+^{2^*_s-1}\leq\frac{C}{|x|^{(N+\tilde{\beta})(2^*_s-1)}},\,\,\,\,|x|\geq1.$$
Letting $\beta=-\bar{\alpha}$, by comparison, we have
$$u_+(x)\leq\frac{C}{|x|^{\frac{N-2s}{2}+\beta}},\,\,\,\,|x|\geq1,$$
where $\beta\in(s,\frac{N-2s}{2})$.
Similarly,
$$(-u)_+(x)\leq\frac{C}{|x|^{\frac{N-2s}{2}+\beta}},\,\,\,\,|x|\geq1.$$

\end{proof}

\begin{prop}\label{propbb3}
Let $u$ be a solution of \eqref{b.1}. Then, we have
$$
|u(x)|\leq\frac{C}{|x|^{\frac{N-2s}{2}-\beta}},\,\,\,\,\forall \,|x|\leq1,
$$
where $\beta$ is given in Lemma \ref{lemb.1}.
\end{prop}
\begin{proof}
Applying the Kelvin transformation
$v(x)=|x|^{2s-N}u\bigl(\frac{x}{|x|^2}\bigr)$, we know that $v$
satisfies
\[
(-\Delta)^s v- \frac{\mu v}{|x|^{2s}}=|v|^{2^{*}_s-2}v,
\quad\forall\;|x|\geq 1,
\]
where $0\leq \mu<\bar{\mu}.$
By Lemma \ref{lemb.1}, we have
$$|v(x)|\leq \frac{C}{|x|^{\frac{N-2s}{2}+\beta}},\quad\forall\;|x|\geq 1.$$
As a result,
$$|u(x)|\leq \frac{C}{|x|^{\frac{N-2s}{2}-\beta}},\quad\forall\;|x|\leq 1.$$
Thus, the result follows.

\end{proof}
\section{\bf A local Pohozaev identity}
In this section, we give a local Pohozaev identity.
For $\mathcal{F}\subset \R^{N+1}_{+}$, we recall that
$\partial_+\mathcal{F}=\big\{z=(x,t)\in\R^{N+1}_{+}:\,(x,t)\in\partial\mathcal{F}\, \text{and}\, t>0\big\}$ and
$\partial_b\mathcal{F}=\partial\mathcal{F}\cap (\R^N\times\{0\})$.  We have the following result.

\begin{prop}\label{propc.1}
Let $E\subset \R^{N+1}_{+}$ and we assume that $\bar{u}$ is a solution of
\begin{equation}\label{d.1}
\begin{cases}
\displaystyle div(t^{1-2s}\nabla \bar{u})=0,  &\text{in}\;E,\\
\mathcal{A}_s(\bar{u})=\ds\frac{\mu}{|x|^{2s}}u+|u|^{p-2}u+au, &\text{on}\;\partial_b E.
\end{cases}
\end{equation}
Then for $\mathcal{F}\subset E,$ there holds
\begin{eqnarray*}
&&\Big(\frac{N}{p}-\frac{N-2s}{2}\Big)\ds\int_{\partial_b\mathcal{F}}|u|^pdx+sa\ds\int_{\partial_b\mathcal{F}}|u|^2dx
+s\mu\ds\int_{\partial_b\mathcal{F}}\frac{x\cdot x_0|u|^2}{|x|^{2s+2}}dx\\
&=&\frac{1}{2}\ds\int_{\partial\partial_b\mathcal{F}}\Big(a+\frac{\mu}{|x|^{2s}}\Big)|u|^2(x-x_0)\cdot\nu_xdS_x
+\frac{1}{p}\ds\int_{\partial\partial_b\mathcal{F}}|u|^p(x-x_0)\nu_xdS_x\\
&&+\ds\int_{\partial_+\mathcal{F}}t^{1-2s}\Big((z-z_0,\nabla\bar{u})\nabla\bar{u}-(z-z_0)\frac{|\nabla\bar{u}|^2}{2},\nu_z\Big)dS_z\\
&&+\frac{N-2s}{2}\ds\int_{\partial_+\mathcal{F}}t^{1-2s}\bar{u}\frac{\partial \bar{u}}{\partial\nu_z}dS_z.
\end{eqnarray*}
\end{prop}

\begin{proof}
Note that
\begin{eqnarray*}
&&
div(t^{1-2s}\nabla \bar{u})(z-z_{0},\nabla \bar{u})\\
&=&
div[t^{1-2s}\nabla \bar{u}(z-z_{0},\nabla \bar{u})]
-(t^{1-2s}\nabla \bar{u},\nabla(z-z_{0},\nabla \bar{u}))\\
&=&
div[t^{1-2s}\nabla \bar{u}(z-z_{0},\nabla \bar{u})]
-t^{1-2s}\Big[(z-z_{0},\nabla\big(\frac{|\nabla \bar{u}|^{2}}{2}\big))
+|\nabla \bar{u}|^{2}\Big]\\
&=&
div\Big[t^{1-2s}\nabla \bar{u}(z-z_{0},\nabla \bar{u})
-t^{1-2s}(z-z_{0})\frac{|\nabla \bar{u}|^{2}}{2}\Big]
+\frac{N-2s}{2}t^{1-2s}|\nabla\bar{u}|^2.
\end{eqnarray*}

Then
$$
div\Big\{t^{1-2s}(z-z_0,\nabla\bar{u})\nabla\bar{u}-t^{1-2s}\frac{|\nabla\bar{u}|^2}{2}(z-z_0)\Big\}
+\frac{N-2s}{2}t^{1-2s}|\nabla\bar{u}|^2=0.
$$
So we find
\begin{equation}\label{d.2}
\begin{array}{ll}
&\ds\int_{\partial_+\mathcal{F}}t^{1-2s}\Big((z-z_0,\nabla\bar{u})\nabla\bar{u}-(z-z_0)\frac{|\nabla\bar{u}|^2}{2},\nu_z\Big)dS_z\vspace{0.2cm}\\
&=-\ds\int_{\partial_b\mathcal{F}}(x-x_0,\nabla_x\bar{u})\mathcal{A}_s\bar{u}dx-\frac{N-2s}{2}\ds\int_{\mathcal{F}}t^{1-2s}|\nabla\bar{u}|^2dxdt.
\end{array}
\end{equation}
By using the fact that $\mathcal{A}_s\bar{u}=\frac{\mu}{|x|^{2s}}u+|u|^{p-2}u+au\,\text{on}\,\partial_b\mathcal{F}$
and performing integration by parts, one has
\begin{equation}\label{d.3}
\begin{array}{ll}
\ds\int_{\partial_b\mathcal{F}}(x-x_0,\nabla_x\bar{u})\mathcal{A}_s\bar{u}dx
&=\ds\frac{-N+2s}{2}\mu\ds\int_{\partial_b\mathcal{F}}\frac{u^2}{|x|^{2s}}dx-s\mu\ds\int_{\partial_b\mathcal{F}}
\frac{x\cdot x_0|u|^2}{|x|^{2s+2}}dx\vspace{0.2cm}\\
&\quad+\ds\frac{\mu}{2}\ds\int_{\partial\partial_b\mathcal{F}}\frac{|u|^2}{|x|^{2s}}(x-x_0)\cdot\nu_xdS_x\vspace{0.2cm}\\
&\quad-\ds\frac{N}{p}\ds\int_{\partial_b\mathcal{F}}|u|^pdx+\frac{1}{p}\ds\int_{\partial\partial_b\mathcal{F}}|u|^p(x-x_0)\cdot\nu_xdS_x\vspace{0.2cm}\\
&\quad-\ds\frac{Na}{2}\ds\int_{\partial_b\mathcal{F}}|u|^2dx+\frac{a}{2}\ds\int_{\partial\partial_b\mathcal{F}}|u|^2(x-x_0)\cdot\nu_xdS_x,
\end{array}
\end{equation}
where we use the fact that
\begin{eqnarray*}
\ds\int_{\partial_b\mathcal{F}}(x-x_0,\nabla_x\bar{u})\frac{u}{|x|^{2s}}dx
&=&\ds\frac{-N+2s}{2}\ds\int_{\partial_b\mathcal{F}}\frac{u^2}{|x|^{2s}}dx-s\ds\int_{\partial_b\mathcal{F}}
\frac{x\cdot x_0|u|^2}{|x|^{2s+2}}dx\\
&&+\ds\frac{1}{2}\ds\int_{\partial\partial_b\mathcal{F}}\frac{|u|^2}{|x|^{2s}}(x-x_0)\cdot\nu_xdS_x,
\end{eqnarray*}
and
\begin{eqnarray*}
\ds\int_{\partial_b\mathcal{F}}(x-x_0,\nabla_x\bar{u})|u|^{p-2}u dx
&=&-\ds\frac{N}{p}\ds\int_{\partial_b\mathcal{F}}|u|^{p}dx
+\ds\frac{1}{p}\ds\int_{\partial\partial_b\mathcal{F}}|u|^p(x-x_0)\cdot\nu_xdS_x.
\end{eqnarray*}
On the other hand,
\begin{equation}\label{d.4}
\begin{array}{ll}
\ds\int_{\mathcal{F}}t^{1-2s}|\nabla\bar{u}|^2dxdt
=\ds\int_{\partial_b\mathcal{F}}\Big(\frac{\mu}{|x|^{2s}}u^2+|u|^p+au^2\Big)dx+\ds\int_{\partial_+\mathcal{F}}t^{1-2s}\bar{u}\frac{\partial \bar{u}}{\partial\nu_z}dS_z.
\end{array}
\end{equation}
Thus, from \eqref{d.2} to \eqref{d.4}, the desired result follows.
\end{proof}

\section{ \bf Decomposition of approximating solutions}

 In this
section, we give a result describing the composition of
approximating solutions bounded in $H^{s}_{0}(\Omega)$ obtained as a
solution of \eqref{1.3} with $\epsilon=\epsilon_{n}$.

\begin{prop}\label{propd.1}
Suppose that $N>6s.$ Let $u_{n}$ be a solution of \eqref{1.3} with
$\epsilon=\epsilon_{n}\rightarrow 0,$ satisfying $\|u_{n}\|\leq C$
for some constant C. Then

(i) $u_{n}$ can be decomposed as
\begin{equation}\label{c.1}
u_{n}=u_{0}+\sum_{j=1}^{m}\rho_{0,\Lambda_{n,j}}(U_{j})
+\sum_{j=m+1}^{h}\rho_{x_{n,j},\Lambda_{n,j}}(U_{j})+\omega_{n},
\end{equation}
where $\omega_{n}\rightarrow 0$ in $H^{s}(\Omega)$, $u_{0}$ is a
solution for \eqref{1.1}. For
$j=1,2,\cdot\cdot\cdot,m,\Lambda_{n,j}\rightarrow\infty$ as $n\rightarrow\infty,$ and $U_{j}$ is a solution
of
\begin{equation}\label{c.2}
(-\Delta)^s u-\mu\frac{u}{|x|^{2s}}=b_{j}|u|^{2^{*}_s-2}u,\,\,\,u\in
D^{s}(\R^{N}),
\end{equation}
for some $b_{j}\in(0,1].$ For
$j=m+1,m+2,\cdot\cdot\cdot,h,x_{n,j}\in\Omega,\Lambda_{n,j}d(x_{n,j},\partial\Omega)\rightarrow\infty,\Lambda_{n,j}|x_{n,j}|\rightarrow\infty$ as $n\rightarrow\infty,$ and $U_{j}$ is a solution
of
\begin{equation}\label{c.2}
(-\Delta)^s u=b_{j}|u|^{2^{*}_s-2}u,\,\,\,u\in
D^{s}(\R^{N}),
\end{equation}
for some $b_{j}\in (0,1].$

(ii) Set $x_{n,i}=0$ for $i=1,2,\cdots,m.$ For $i,j=1,2,\cdot\cdot\cdot,h,$ if $i\neq j$, then as
$n\rightarrow\infty,$
\begin{equation}\label{c1.3}
\frac{\Lambda_{n,j}}{\Lambda_{n,i}}+\frac{\Lambda_{n,i}}{\Lambda_{n,j}}
+\Lambda_{n,j}\Lambda_{n,i}|x_{n,i}-x_{n,j}|^{2}\rightarrow\infty.
\end{equation}
\end{prop}

\begin{proof}
The proof is similar to \cite{sa,cs,cps,cs1} and we omit the details.
\end{proof}

{\bf Acknowledgements:}
  This paper was partially supported by NSFC (No.11301204; No.11371159; No. 11561043), self-determined research funds of CCNU from
  the colleges' basic research and operation of MOE (CCNU14A05036).

\end{document}